\documentclass[11pt,a4paper]{amsart}
\usepackage{amssymb}
\usepackage{graphicx}
\usepackage{psfrag}

\newcommand{\eref}[1]{(\ref{#1})}

\theoremstyle{plain}

\newtheorem{Thm}{Theorem}[section]
\newtheorem{Prop}[Thm]{Proposition}
\newtheorem{Lem}[Thm]{Lemma}

\newtheorem*{Claim}{Claim}
\newtheorem{Cor}[Thm]{Corollary}

\theoremstyle{remark}
\newtheorem{Rem}[Thm]{Remark}

\newtheorem{Exa}[Thm]{Example}

\def\real{{\mathbb R}}
\def\rational{{\mathbb Q}}

\def\C{\mathcal{C}}
\def\O{\mathcal{C}}
\def\P{\mathcal{P}}
\def\S{\mathcal{U}}
\def\U{\mathcal{U}}
\def\ome{\tilde{\om}}
\def\a{a_J}
\def\ta{\tau}

\def\I{I}
\def\i{[0,\delta]}
\def\ii{[0,\vep]}
\def\T{\tilde{T}}

\def\const{\operatorname{const}}
\def\closure{\operatorname{closure}}
\def\ind{\operatorname{symb}}
\def\dist{\operatorname{dist}}

\def\inte{\operatorname{int}}
\def\Lip{\operatorname{Lip}}
\def\max{\operatorname{max}}
\def\min{\operatorname{min}}
\def\mod{\operatorname{mod}}

\def\sign{\operatorname{sign}}
\def\supp{\operatorname{supp}}

\def\al{\alpha}
\def\be{\beta}
\def\vep{\varepsilon}

\def\om{\omega}
\def\la{\lambda}
\def\La{\Lambda}

\def\ka{\kappa}

\begin{document}

\title[Typical points]
{Typical points for one-parameter families of piecewise expanding maps of the interval}
\author[Daniel Schnellmann]{}
\subjclass{Primary: 37E05; 37A05; Secondary: 37D20.}

\email{daniel.schnellmann@ens.fr}

\thanks{Research supported by the Swedish Research Council and grant KAW 2005.0098 from the Knut and 
Alice Wallenberg Foundation.}

\maketitle

\centerline{\scshape Daniel Schnellmann}
\medskip
{\footnotesize
   \centerline{Ecole Normale Sup\'erieure}
   \centerline{D\'epartment de math\'ematiques et applications (DMA)} 
   \centerline{45 rue d'Ulm}
   \centerline{75230 Paris cedex 05, France}
} 

\begin{abstract}
Let $\I\subset\real$ be an interval and $T_a:[0,1]\to[0,1]$, $a\in\I$, a one-parameter family of 
piecewise expanding maps such that for each $a\in\I$ the map $T_a$ admits a unique absolutely continuous 
invariant probability measure $\mu_a$. We establish sufficient conditions on such a one-parameter family 
such that a given point $x\in[0,1]$ is typical for $\mu_a$ for a full Lebesgue measure set of parameters 
$a$, i.e., 
$$
\frac{1}{n}\sum_{i=0}^{n-1}\delta_{T_a^i(x)}
\overset{\text{weak-}*}{\longrightarrow}\mu_a,\qquad\text{as}\ n\to\infty,
$$
for Lebesgue almost every $a\in\I$. In particular, we consider $C^{1,1}(L)$-versions of $\be$-transformations, 
piecewise expanding unimodal maps, and Markov structure preserving one-parameter families.
For families of piecewise expanding unimodal maps we show that the turning point is almost surely typical 
whenever the family is transversal.
\end{abstract}



\section{Introduction}
Let $\I\subset\real$ be an interval and $T_a:[0,1]\to[0,1]$, $a\in\I$, a 
one-parameter family of maps of the unit interval such that, for every $a\in\I$, $T_a$ is piecewise $C^2$ and 
$\inf_{x\in[0,1]}|\partial_x T_a(x)|\ge\la>1$, where $\la$ is independent on $a$. 
Assume that, for all $a\in\I$, $T_a$ admits a unique (hence ergodic) absolutely continuous 
invariant probability measure (a.c.i.p.) $\mu_a$. According to \cite{lasota} and \cite{liyorke}, 
for Lebesgue almost every $x\in[0,1]$, some iteration of $x$ by $T_a$ is contained in the support of $\mu_a$.
From Birkhoff's ergodic theorem we derive that Lebesgue almost every point $x\in[0,1]$ 
is {\em typical\/} for $\mu_a$, i.e., 
$$
\frac{1}{n}\sum_{i=0}^{n-1}\delta_{T_a^i(x)}
\overset{\text{weak-}*}{\longrightarrow}\mu_a,\qquad\text{as}\ n\to\infty.
$$
In this paper we are interested in the question if the same kind of statement holds in 
the parameter space, i.e., if a chosen point $x\in[0,1]$ is typical for $\mu_a$ for 
Lebesgue a.e. $a\in\I$, or more general, if, for some given $C^1$ map 
$X:\I\to[0,1]$, $X(a)$ is typical for $\mu_a$ for Lebesgue a.e. $a$ in $\I$.  In Section~\ref{s.uniform} we 
try to establish sufficient conditions on a one-parameter family such that the following statement is true.

\begin{center}
{\em For Lebesgue a.e. $a\in\I$, $X(a)$ is typical for $\mu_a$.\/}
\end{center}

The method we use in this paper is a dynamical one. It is essentially inspired 
by the result of Benedicks and Carleson \cite{bc} where they prove that for the quadratic 
family $f_a(x)=1-ax^2$ on $(-1,1)$ there is a set $\Delta_{\infty}\subset(1,2)$ 
of $a$-values of positive Lebesgue measure for which $f_a$ admits almost surely an a.c.i.p. and 
for which the critical point is typical with respect to this a.c.i.p. 
The main tool in their work is to switch from the parameter space to the dynamical 
interval by showing that the $a$-derivative $\partial_a f_a^j(1)$ is comparable to the 
$x$-derivative $\partial_x f_a^j(1)$. This will also be the essence of the basic 
condition on our one-parameter family $T_a$ with an associated map $X$, i.e., we require that the $a$- and the 
$x$-derivatives of $T_a^j(X(a))$ are comparable  (see condition~(I) below).

Some typicality results related to this paper can be found in \cite{schmeling}, \cite{bruin}, and 
\cite{fp}. The one-parameter families $T_a$, $a\in\I$, considered in these papers have in common that their 
slopes are constant for a fixed parameter value, i.e., for each $a\in\I$ there is a constant $\la_a>1$ 
such that $|T_a'|\equiv\la_a$ on $[0,1]$. 
The advantage of our method and the main novelty of this paper is that we can drop this restriction and, thus, we 
are able to consider much more general families. 
This paper consists mainly of two parts. In the first part, which corresponds to 
Sections~\ref{s.uniform}-\ref{s.keylemmas}, we establish a general criteria for typicality.
In the second part, which corresponds to Sections~\ref{s.beta}-\ref{s.markov}, 
we apply this criteria to several well-studied one-parameter families and derive various 
typicality results for these families. Some of the results are presented in the following sections of this introduction. 

We will shortly give a motivation and an overview of our criteria for typicality. Let $B\subset[0,1]$ be a (small) interval and 
set $x_j(a)=T_a^j(X(a))$, $a\in\I$, i.e., $x_j(a)$ is the forward iteration by $T_a^j$ of the points we are interested in. 
For $h\ge1$ fixed, the main estimate to be established in the method we apply is roughly of the form:
\begin{equation}
\label{eq.intromain1}
\frac{1}{|I|}|\{a\in\I\ ;\ x_{j_1}(a)\in B, ... ,x_{j_h}(a)\in B\}|\le(C|B|)^h,
\end{equation}
where $1\le j_1<...<j_h\le n$ ($n$ large) are $h$ integers with large ($\ge\sqrt{n}$) gaps between each other and $C\ge1$ is some constant. 
Such an estimate is easier to establish on the phase space for a fixed map $T_a$ in the family, i.e., it is easier to verify the estimate
\begin{equation}
\label{eq.intromain2}
|\{x\in[0,1]\ ;\ T_a^{j_1}(x)\in B, ... ,T_a^{j_h}(x)\in B\}|\le(C|B|)^h.
\end{equation}
(See also inequality~\eref{m.exact}.)
Hence, in order to prove \eref{eq.intromain1}, the main idea in the first part of this paper is to compare sets 
in the parameter space $\I$ with sets in the phase space $[0,1]$. This will be possible if the following three conditions, conditions~(I)-(III), 
are satisfied. The first two conditions are rather natural while the last condition is a bit technical and restrictive. 
(However, as we will see in Sections~\ref{s.beta}-\ref{s.markov} these conditions are satisfied for a broad class of important 
one-parameter families of piecewise expanding maps.) 
Condition~(I) roughly says 
that $T_a^j$ and $x_j$ are comparable, i.e., there exists a constant $C\ge1$ such that
\begin{equation}
\label{eq.intromain0}
C^{-1}\le\frac{|D_ax_j(a)|}{|\partial_xT_a^j(X(a))|}\le C,
\end{equation}
for all $j\ge1$, and $a\in\I$ for which the derivatives are defined.  This is a well-known condition for one-parameter families 
of maps on the interval. In the case of piecewise expanding unimodal maps if one chooses $X$ to be equal to the 
turning point (or some iteration of it to make the $x$-derivative well-defined), 
this condition is equivalent to saying that the family is transversal (cf. Lemma~\ref{l.nondegenerate}). 
In fact, in order that \eref{eq.intromain0} holds for all $j\ge1$ it is enough to require that 
the map $x_j:\I\to[0,1]$ has a sufficiently high initial expansion for some $j<\infty$ (see Lemma~\ref{l.startcalc}), 
which makes condition~(I) easy to check.
Condition~(II) requires that the density for $\mu_a$ is uniformly in $a$ bounded above and below away from $0$. 
This ensures that there is a constant $C\ge1$ such that for all $a\in\I$ we have the estimate
\begin{equation}
\label{eq.intromain21}
|\{x\in\supp(\mu_a)\ ;\ T_a^j(x)\in B\}|\le C|B|,
\end{equation}
for all $j\ge1$ (cf. inequality \eref{eq.perronfrobenius}).
Condition~(III) requires that there is a kind of order relation in the one-parameter 
family in the sense that for each two parameter values $a,a'\in\I$ satisfying $a<a'$ the following holds. The symbolic 
dynamics of $T_a$ is contained in the symbolic dynamics of $T_{a'}$ and, furthermore, if $\om$ is a (maximal) interval of 
monotonicity for $T_a^j$ and $\om'$ is the corresponding (maximal) interval of monotonicity for $T_{a'}^j$ (i.e., $\om'$ 
has the same combinatorics as $\om$ up to the $j$-th iteration), then the following holds. 
The minimal distance of points in the image $T_a^j(\om)$ to points in the image $T_{a'}^j(\om')$ is bounded from 
above by $|a'-a|$, i.e., when the parameter are close then the images are close. Further, the size of $T_a^j(\om)$ is 
bounded above by the size of $T_{a'}^j(\om')$.

We shortly point out how to get an estimate as in \eref{eq.intromain1}. 
If conditions~(I)-(III) are satisfied we first divide the parameter interval into intervals $J\subset\I$ of length $1/n$ ($n$ large) 
in order to have good distortion estimates when switching from maps on the parameter interval to maps on the dynamical 
interval. On each such interval $J$ we will establish roughly the estimate 
\begin{multline}
\label{eq.intromain3}
|\{a\in J\ ;\ x_{j_1}(a)\in B, ... ,x_{j_h}(a)\in B\}|\\
\lesssim |\{x\in J_x\ ;\ T_{\a}^{j_1}(x)\in2B, ... ,T_{\a}^{j_h}(x)\in2B\}|\le\frac1n(C|B|)^h,
\end{multline}
where $\a$ is the right boundary point of $J$, $J_x\subset[0,1]$ is an interval of approximately
size $1/n$ containing the image $X(J)$, and $2B$ is the interval twice as long as $B$ and having the same midpoint as $B$. 
The first inequality in \eref{eq.intromain3} is essentially due to condition~(I) and (III) where the main ingredient is 
Lemma~\ref{l.distortion1}. 
Using the estimate \eref{eq.intromain21}, the last inequality in \eref{eq.intromain3} 
(or, similarly, inequality \eref{eq.intromain2})
is straightforward to verify whenever the family has the property that the image by 
$T_a^j$ of each (maximal) monotonicity interval for $T_a^j$ has size close to $1$ (this is, e.g., the case when the family 
preserves a Markov structure; cf. Section~\ref{s.markov}). However, in the general case there are many 
monotonicity intervals with very small images which makes the proof more technical. A sufficient 
upper bound for the measure of exceptionally small monotonicity intervals is established in 
Lemma~\ref{l.remainderintervals} where an important ingredient is the 
$\sqrt{n}$ gap between the $j_i$'s which makes it for too small intervals possible to 'recover'.

\subsection{$\beta$-transformations}
The example in Section~\ref{s.beta} is a $C^{1,1}(L)$-version of the $\be$-trans\-for\-ma\-tion. By saying that a 
function is $C^{1,1}(L)$, we mean that it is $C^1$ and its derivative is in $\Lip(L)$, i.e., its derivative is Lipschitz 
continuous with Lipschitz constant bounded above by $L$. 
For a sequence $0=b_0<b_1<...$ of real numbers such that $b_k\to\infty$ as $k\to\infty$ and a constant $L>0$, 
let $T:[0,\infty)\to[0,1]$ be a right continuous function which is $C^{1,1}(L)$ on each interval 
$[b_k,b_{k+1})$, $k\ge0$. Furthermore, we assume that:
\begin{itemize}
\item
$T(b_k)=0$ for each $k\ge0$.
\item 
For each $a>1$,
$$
1<\inf_{x\in[0,1]}\partial_xT(ax)\quad\text{and}\quad\sup_{x\in[0,1]}\partial_xT(ax)<\infty.
$$
\end{itemize}
See Figure~\ref{f.beta}.

\begin{figure}[htb]
\centering
\includegraphics[width=8cm]{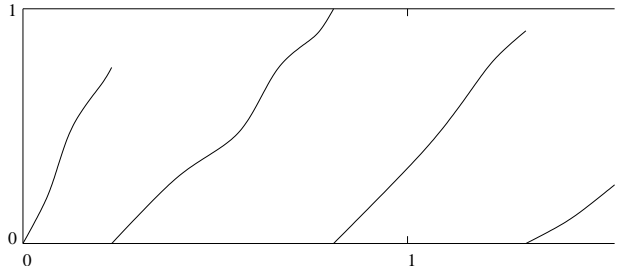}
\caption{A possible beginning of a graph for $T:[0,\infty)\to[0,1]$.}
\label{f.beta}
\end{figure}

Given a map $T$ as above, we obtain a $C^{1,1}(L)$-version of the $\be$-transformation 
$T_a:[0,1]\to[0,1]$, $a>1$, by defining $T_a(x)=T(ax)$, $x\in[0,1]$. 
As we will see in Section~\ref{s.beta}, for each $a>1$, $T_a$ admits a unique a.c.i.p. $\mu_a$, 
and there are many functions $X$ for which we have almost sure typicality.

\begin{Thm}
\label{t.intro2}
If $X:(1,\infty)\to(0,1]$ is $C^1$ and $X'(a)\ge0$, then $X(a)$ is typical for $\mu_a$ for Lebesgue a.e. $a>1$.
\end{Thm}

If we choose $X(a)=b_1/a$ then $X'(a)<0$ and $T_a^j(X(a))=0$ for all $j\ge1$. Hence, if the condition $X'(a)\ge0$ 
in Theorem~\ref{t.intro2} is not satisfied, we cannot any longer guarantee almost sure typicality for the a.c.i.p.
For an illustration of some curves on which we have a.s. typicality see Figure~\ref{f.birkhoff} 
(when $a$ is fixed, we can apply Birkhoff's ergodic theorem and get a.s. 
typicality on the associated vertical line). Observe that if we choose $T:[0,\infty)\to[0,1]$ 
by $T(x)=x\mod1$, then $T_a(x)=ax\mod 1$ is the usual $\be$-transformation. 
Theorem~\ref{t.intro2} generalizes a result due to 
Schmeling \cite{schmeling} where it is shown that for the usual $\beta$-transformation 
the point $1$ is typical for the associated a.c.i.p. for Lebesgue a.e. $a>1$.

\begin{figure}[htb]
\centering
\includegraphics[width=8cm]{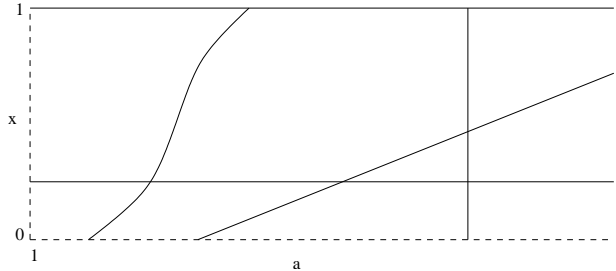}
\caption{Lines and curves on which we have a.s. typicality for the $C^{1,1}(L)$-version of 
the $\beta$-transformation.}
\label{f.birkhoff}
\end{figure}

\subsection{Unimodal maps}
In Section~\ref{s.unimodal} we investigate one-parameter families of piecewise expanding unimodal maps.
A map $T:[0,1]\to[0,1]$ is a {\em piecewise expanding unimodal map} if it is continuous and if there exists a turning point 
$c\in(0,1)$ such that $T|_{[0,c]}$ and $T|_{[c,1]}$ are $C^{1,1}(L)$, $\inf_{x\in[0,c]}T'(x)>1$ and $\sup_{x\in[c,1]}T'(x)<-1$, 
and $T(c)=1$ and $T^2(c)=0$. (In Section~\ref{s.unimodal} we will use a slightly different representation of piecewise 
unimodal maps which will be more convenient to work with.) 
By \cite{liyorke} (or \cite{wong}), since $T$ is only at one point not $C^{1,1}(L)$, there exists a unique a.c.i.p. $\mu$. We call 
$T$ mixing, if it is topologically mixing on $[0,1]$. (This implies that $\supp(\mu)=[0,1]$.)
Let $\I\subset\real$ be a finite closed interval. For  
a one-parameter family of piecewise expanding unimodal maps $T_a$, $a\in\I$, we make some natural requirements on 
the parameter dependency as, e.g., the following (cf. Section~\ref{ss.prel} properties~(i)-(iii)). For each $x\in[0,1]$ the 
map $a\mapsto T_a(x)$ is $\Lip(L)$ on $\I$ (in particular, this implies that the turning point $c(a)$ of $T_a$ is Lipschitz 
continuous in $a$), and if $J\subset\I$ is an interval on which $x\neq c(a)$, 
then $a\mapsto T_a(x)$ is $C^1(J)$ and $a\mapsto\partial_xT_a(x)$ is $\Lip(L)$ on $J$. The main result in 
Section~\ref{s.unimodal} can be stated as follows.

\begin{Thm}
\label{t.uniintro}
If $T_a$, $a\in\I$, is a one-parameter family of mixing piecewise expanding unimodal maps such that for some 
$j_0\ge3$,
$$
|D_aT_a^{j_0}(c(a))|>\frac{\sup_{a\in\I}\sup_{x\in[0,1]}|\partial_aT_a(x)|}{\inf_{a\in\I}\inf_{x\in[0,1]}|\partial_xT_a(x)|-1},
$$
for all but finitely many $a\in\I$ (i.e., the set of $a\in\I$ for which $a\mapsto T_a^{j_0}(c(a))$ is not differentiable is finite),
then the turning point $c(a)$ is typical for the a.c.i.p. for $T_a$, for almost every $a\in\I$.
\end{Thm}

Other ways of stating Theorem~\ref{t.uniintro} are: If there exists a $j_0\ge3$ such that the map $X(a)=T_a^{j_0}(c(a))$ 
satisfies condition~(I) (see Section~\ref{ss.main}) then the turning point is almost surely typical; or if 
the family is non-degenerate (or transversal) in each point $a\in\I$ (see Section~\ref{s.unimodal}) then 
the turning point is almost surely typical. In Section~\ref{s.unimodal} we will state a more local version 
of Theorem~\ref{t.uniintro}.

In Section~\ref{s.tent} we will apply Theorem~\ref{t.uniintro} to unimodal maps 
with slopes constant to the left and to the right of the turning point, the so called {\em skew tent maps}. 
Let these slopes be $\al$ and $-\be$ where $\al,\be>1$ and denote the corresponding skew tent map by 
$T_{\al,\be}$. (In order that $T_{\al,\be}$ maps the unit interval into itself, 
we have also to assume that $\al^{-1}+\be^{-1}\ge1$.)
Fix two points $(\al_0,\be_0)$ and $(\al_1,\be_1)$ in the set 
$\{(\al,\be)\ ;\ \al,\be>1\ \text{and}\ \al^{-1}+\be^{-1}\ge1\}$ such that 
$\al_1\ge\al_0$, $\be_1\ge\be_0$, and at least one of these two inequalities is sharp. 
Let
$$
\al:[0,1]\to[\al_0,\al_1]\quad\text{and}\quad\be:[0,1]\to[\be_0,\be_1]
$$
be functions in $C^1([0,1])$ such that $(\al(0),\be(0))=(\al_0,\be_0)$, $(\al(1),\be(1))=(\al_1,\be_1)$, 
and, for all $a\in[0,1]$, if $\al_0\neq\al_1$, then $\al'(a)>0$ and if $\be_0\neq\be_1$, then $\be'(a)>0$.
Consider the one-parameter family $T_a$, 
$a\in[0,1]$, where $T_a:[0,1]\to[0,1]$ is the skew tent map defined by $T_a=T_{\al(a),\be(a)}$. 
The main result of Section~\ref{s.tent} is the following.

\begin{Thm}
\label{t.intro3}
The turning point for the skew tent map $T_a$ is typical 
for $\mu_a$ for Lebesgue a.e. $a\in[0,1]$.
\end{Thm}

Theorem~\ref{t.intro3} generalizes a result due to Bruin \cite{bruin} where almost sure typicality is shown 
for the turning point of symmetric tent maps (i.e., when $\al(a)\equiv\be(a)$). (See also \cite{bm}.)

\subsection{One-parameter families of Markov maps}
In Section~\ref{s.markov} we consider one-pa\-ra\-me\-ter families preserving a certain Markov structure. 
A simple example for such a family are the maps $T_a:[0,1]\to[0,1]$ defined by
\begin{equation*}
T_a(x)=\left\{\begin{array}{ll}
              \frac{x}{a}\quad&\text{if}\ x<a,\\
              \frac{x-a}{1-a}\quad&\text{otherwise},
              \end{array}\right.
\end{equation*}
where the parameter $a\in(0,1)$. See Figure~\ref{f.markov}. By \cite{liyorke}, since this map has only one 
point of discontinuity, it admits a unique a.c.i.p. $\mu_a$ (which coincides in this case with the Lebesgue measure on 
$[0,1]$). In Example~\ref{ex.markov} in Section~\ref{s.markov} we will show the following. 

\begin{Prop}
\label{p.intro4}
If $X:(0,1)\to(0,1)$ is a $C^1$ map such that $X'(a)\le0$, then
$X(a)$ is typical for $\mu_a$ for a.e. parameter $a\in(0,1)$.
\end{Prop}

Observe that if $X(a)=p_a$ where $p_a$ is the unique point of periodicity $2$ in the interval $(0,a)$, then 
$X'(a)>0$ and $X(a)$ is not typical for $\mu_a$ for any $a\in\I$. 
The very simple structure of the family in this section makes it a good candidate for serving the reader 
as a model along the paper. Example~\ref{ex.markov} in Section~\ref{s.markov} is formulated slightly more generally by 
composing $T_a$ with a $C^{1,1}(L)$ homeomorphism $g:[0,1]\to[0,1]$.

\begin{figure}[htb]
\centering
\includegraphics[width=8cm]{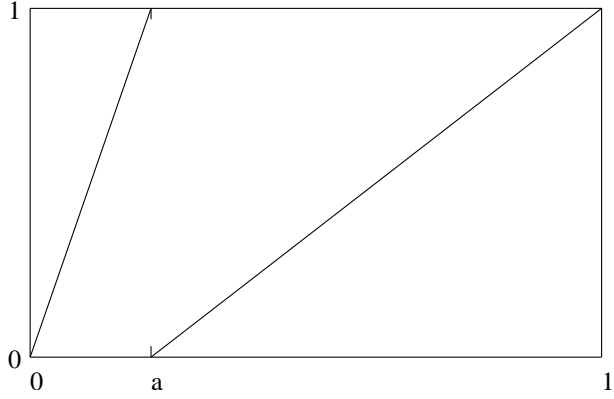}
\caption{A Markov structure preserving one-parameter family $T_a$ where $a\in(0,1)$.}
\label{f.markov}
\end{figure}


\section{Piecewise expanding one-parameter families}
\label{s.uniform}

\subsection{Preliminaries}
\label{ss.prel}

In this section we introduce the basic notation and put up a general model for 
one-parameter families of piecewise expanding maps of the unit interval. 
A map $T:[0,1]\to\real$ will be called {\em piecewise $C^{1,1}(L)$\/} if there exists a partition 
$0=b_0<b_1<...<b_p=1$ of the unit interval such that for each $1\le k\le p$ the restriction 
of $T$ to the open interval $(b_{k-1},b_k)$ is a $C^{1,1}(L)$ function. Observe that, by the Lipschitz 
property, it follows that 
$T$ restricted to $(b_{k-1},b_k)$ can be extended to the closed interval $[b_{k-1},b_k]$ as a 
$C^{1,1}(L)$ function. Let $\I\subset\real$ be an interval of finite length and $T_a:[0,1]\to[0,1]$, 
$a\in\I$, a one-parameter family of piecewise $C^{1,1}(L)$ maps where the Lipschitz constant $0<L<\infty$ is 
independent on the choice of the parameter $a$. 
We assume that there are real numbers $1<\la\le\La<\infty$ such that for every $a\in\I$, 
\begin{equation}
\label{eq.la}
\la\leq\inf_{x\in[0,1]}|\partial_x T_a(x)|\quad\text{and}\quad\sup_{x\in[0,1]}|\partial_x T_a(x)|\leq\La.
\end{equation}  
Let $0=b_0(a)<b_1(a)<...<b_{p(a)}(a)=1$ be the partition of the unit interval associated to $T_a$. 
We make the following natural assumption on the parameter dependence. 

\begin{itemize}
\item[(i)] 
The number of monotonicity intervals for the $T_a$'s is constant, i.e., $p(a)\equiv p_0$, 
and the partition points $b_k(a)$, $0\le k\le p_0$, are $\Lip(L)$ on $\I$. 
Furthermore, there is a constant $\delta_0>0$ such that 
$$
b_k(a)-b_{k-1}(a)\ge\delta_0,
$$
for all $1\le k\le p_0$ and $a\in\I$.
\item[(ii)]
If $x\in[0,1]$ and $J\subset\I$ is a parameter interval 
such that $b_k(a)\neq x$, for all $a\in J$ and $0\le k\le p_0$, then $a\mapsto T_a(x)$ is $C^1(J)$ and 
both maps $a\mapsto T_a(x)$ and $a\mapsto\partial_x T_a(x)$ are $\Lip(L)$ on $J$.
\end{itemize}

In the sequel, instead of referring to \cite{lasota} and \cite{liyorke}, we will refer to a paper by 
S. Wong \cite{wong} who extended the results in \cite{lasota} and \cite{liyorke} on piecewise 
$C^2$ maps to a broader class of maps containing also piecewise $C^{1,1}(L)$ maps.
For a fixed $a\in\I$, by \cite{wong}, the number of ergodic a.c.i.p. for $T_a$ is finite and the support of an ergodic a.c.i.p. is 
a finite union of intervals. Since we are always interested in only one ergodic a.c.i.p., we can 
without loss of generality assume that for each $T_a$, $a\in\I$, there is a unique (hence ergodic) a.c.i.p. which we denote by $\mu_a$.
Let $K(a)=\supp(\mu_a)$. By \cite{wong}, for Lebesgue a.e. $x\in[0,1]$, 
the accumulation points of the forward orbit of $x$ is identical with $K(a)$, i.e.,
\begin{equation}
\label{eq.propertyK}
K(a)=\bigcap_{N=1}^{\infty}\overline{\{T_a^n(x)\}_{n=N}^{\infty}}.
\end{equation}
For $a\in\I$, let 
$$
\{D_1(a),...,D_{p_1(a)}(a)\}=\{\text{connected components of $K(a)\setminus\{b_0(a),...,b_{p_0}(a)\}$}\}, 
$$
i.e., the $D_k(a)$'s are the monotonicity intervals for $T_a:K(a)\to K(a)$. We assume the following.
\begin{itemize}
\item[(iii)]
The number of $D_k(a)$'s is constant in $a$, i.e., $p_1(a)\equiv p_1$ for all $a\in\I$. The boundary points of 
$D_k(a)$, $1\le k\le p_1$, change continuously in $a\in\I$.
\end{itemize}

\subsection{Partitions}
\label{ss.partitions}
For a fixed parameter value $a\in\I$, we denote by $\P_j(a)$, $j\ge1$, the partition on the dynamical interval 
consisting of the maximal open intervals of smooth monotonicity for the map $T_a^j: K(a)\to K(a)$. 
More precisely, $\P_j(a)$ denotes the set of open intervals $\om\subset K(a)$ such that 
$T_a^j:\om\to K(a)$ is $C^{1,1}(L)$ and $\om$ is maximal, i.e., for every other open interval $\ome\subset K(a)$ 
with $\om\subsetneq\ome$, $T_a^j:\ome\to K(a)$ is no longer $C^{1,1}(L)$. 
Clearly, the elements of $\P_1(a)$ are the interior of the intervals $D_k(a)$, $1\le k\le p_1$.
For an interval $H\subset K(a)$, we denote by $\P_j(a)|H$ the restriction 
of $\P_j(a)$ to the set $H$. 
For a set $J\subset K(a)$, for which there exists, $1\le k\le p_1$, such that $J\subset D_k(a)$, 
we denote by $\ind_a(J)$ the index (or symbol) $k$.

We will define similar partitions on the parameter interval $\I$. 
Let $X:\I\to[0,1]$ be a piecewise $C^1$ map from the parameter interval $\I$ into the dynamical interval $[0,1]$. 
The points $X(a)$, $a\in\I$, will be our candidates for typical points. The forward orbit of a point $X(a)$ under the map $T_a$ 
we denote as 
$$
x_j(a):=T_a^j(X(a)),\qquad\quad j\ge0.
$$

\begin{Rem}
\label{rem.critical}
Since a lot of information for the dynamics of $T_a$ is contained in the forward orbits of the 
partition points $b_k(a)$, $0\le k\le p_0$, it is of interest to know how the forward orbits of these points are distributed. 
Hence, an evident choice of the map $X$ would be 
$$
X(a)=\lim_{\begin{subarray}{c}
           x\to b_k(a) \\
           x\in\omega
           \end{subarray}}
     T_a(x),
$$
where $\omega\in\P_1(a)$ is an interval adjacent to $b_k(a)$.
\end{Rem}

Let $J\subset\I$ be an open interval (or a finite union of open intervals) and let $\C$ denote the finite number of $a$ values in 
which $X$ is not differentiable. By $\P_j|J$, $j\ge1$, we denote the partition consisting 
of all open intervals $\om$ in $J\setminus\C$ such that 
for each $0\le i<j$, $x_i(a)\in K(a)\setminus\{b_0(a),...,b_{p_0}(a)\}$, for all $a\in\om$, and such that $\om$ is maximal, 
i.e., for every other open interval $\ome\subset J$ with $\om\subsetneq\ome$, there exist 
$a\in\ome$ and $0\le i<j$ such that $x_i(a)\in\{b_0(a),...,b_{p_0}(a)\}$. 
Observe that this partition might be empty. This is, e.g., the case when $X(a)\notin K(a)$ for all $a\in\I$ or when 
$T_a$ is the usual $\beta$-transformation and the map $X$ is chosen to be equivalently equal to $0$. 
However, such trivial situations are excluded by condition~(I) formulated in the next Section~\ref{ss.main}. Knowing 
that condition~(I) is satisfied, then the partition $\P_j|J$, $j\ge1$, can be thought of as the set of the 
(maximal) intervals of smooth monotonicity for $x_j:J\to[0,1]$ 
(in order that this is true one should also assume that $|x_j'(a)|>L$ for all $a$ contained in a partition element of $\P_j|J$).
We set $\P_0|J=J$ (and we will write $\P_j|\I$ instead of $\P_j|\inte(\I)$).
If for an interval $J'$ in the parameter space and for some integer $j\ge0$ the 
symbol $\ind_a(x_j(a))$ exists for all $a\in J'$, then it is constant and we denote this symbol by $\ind(x_j(J'))$.
Finally, in view of condition~(I) below, observe that if a parameter $a\in\I$ is contained in an element of 
$\P_j|\I$, $j\ge1$, then also the point $X(a)(=x_0(a))$ is contained in an element of $\P_j(a)$ which implies 
that $T_a^j$ is differentiable in $X(a)$.

\subsection{Main statement}
\label{ss.main}
In this section we will state our main result, Theorem~\ref{t.main}. Let $n$ be large. In order to have good distortion 
estimates we will, in the proof of Theorem~\ref{t.main}, split up the interval $\I$ into smaller intervals 
$J\subset\I$ of size $1/n$. The main idea in this paper is to switch from the map
$x_j:J\to[0,1]$, $j\le n$, to the map $T_a^j:J_x\to[0,1]$ where $a$ is the right boundary point of $J$ and $J_x$ is an 
interval of size $\approx 1/n$ oriented around $X(J)$ (assume $X:J\to[0,1]$ has no discontinuities). 
Since the dynamics of the map $T_a$ is well-understood, we derive similar dynamical properties for the map $x_j$, 
which then can be used to prove Theorem~\ref{t.main}. To be able to
switch from $x_j$ to $T_a^j$, we put three conditions, conditions~(I)-(III), on our one-parameter family and on the 
map $X$ associated to it.

In condition (I) we require that the derivatives of $x_j$ and
$T_a^j$ are comparable. This is the very basic assumption in this
paper. Of course, the choice of the map $X:\I\to[0,1]$ plays here an
important role. If, e.g., for every parameter $a\in\I$, $X(a)$ is a
periodic point for the map $T_a$, then $x_j$ will have bounded
derivatives and the dynamics of $x_j$ is completely different from the
dynamics of $T_a$. Henceforth, we will use the notations 
$$
T_a'(x)=\partial_xT_a(x)\qquad\text{and}\qquad x_j'(a)=D_a x_j(a),\quad j\ge1.
$$ 

\begin{itemize}
\item[(I)]
There is a constant $C_0\ge1$ such that  
for $\om\in\P_j|\I$, $j\ge1$, we have 
$$
\frac{1}{C_0}\le\left|\frac{x_j'(a)}{T_a^j\,'(X(a))}\right|\le C_0,
$$
for all $a\in\om$. Furthermore, the number of $a\in\I$, which are not contained in any element $\om\in\P_j|I$ is finite.
\end{itemize}

Given a $C^1$ map $Y:\I\to[0,1]$, as the following basic lemma asserts, to verify that there exists $j_0\ge0$ such that 
condition~(I) is satisfied for the map $X(a)=T_a^{j_0}(Y(a))$ it is sufficient to compute the $a$-derivative of $T_a^j(Y(a))$ 
for a finite number of $j$'s. This makes it easy to check condition~(I) numerically. The proof of this lemma is given in 
Section~\ref{ss.lem22}.

\begin{Lem}
\label{l.startcalc}
Assume that the parameter interval $\I$ is closed, let $Y:\I\to[0,1]$ be $C^1$, and denote $y_j(a)=T_a^j(Y(a))$. 
If there exists $j_0\ge0$ such that $y_{j_0}(a)\in K(a)$, $a\in\I$, 
and $y_{j_0}$ is piecewise $C^1$ (with finitely many pieces) such that
\begin{equation}
\label{eq.startcomp}
\inf_{a\in\I}|y_{j_0}'(a)|\ge\frac{\sup_{a\in\I}\sup_{x\in K(a)}|\partial_aT_a(x)|}{\la-1}+2L,
\end{equation}
then condition~(I) is satisfied for the map $X(a)=T_a^{j_0}(a)$.
\end{Lem}

We turn to condition (II). For $a\in\I$, let $\varphi_a$ denote the density for $\mu_a$. We assume that 
$\varphi_a$ is uniformly in $a$ bounded away from $0$. 

\begin{itemize}
\item[(II)]
There exists a constant $C_1\ge1$ such that, for each $a\in\I$,
$$
\frac1{C_1}\le\varphi_a(x)\le C_1,
$$
for $\mu_a$ almost every $x\in[0,1]$. 
\end{itemize}

\begin{Rem}
\label{r.tomas}
In fact, instead of the uniformity of the constant $C_1$ one could $C_1$ allow to depend 
measurably on the parameter $a\in\I$. In Proposition~\ref{p.main} and its proof below, one could then consider 
closed sets $\I_\vep\subset\I$ with $|\I\setminus\I_\vep|\le\vep$ on which one has by Lusin uniformity of $C_1$. 
The proof of Proposition~\ref{p.main} restricted to such a set $\I_\vep$ could then be 
adapted by observing that for the intervals $J$ in the proof it is only important that the right boundary point of 
$J$ lies in $\I_\vep$. However, in all the examples we are considering in this paper the uniformity of $C_1$ is 
easy to establish.
\end{Rem}

While conditions~(I) and (II) are very natural requirements the following condition~(III) is more restrictive. 
However, as we will see in Sections~\ref{s.beta}-\ref{s.markov}, condition~(III) is satisfied for 
important families of piecewise expanding maps. In particular, in Section~\ref{s.unimodal} we will see that condition~(III) is 
satisfied for all non-degenerate families of piecewise expanding unimodal maps. 
For two non-empty sets $A_1,A_2\subset\real$, $\dist(A_1,A_2)$ denotes the infimum of the distances $|a_1-a_2|$ over all
possible points $a_1\in A_1$ and $a_2\in A_2$.

\begin{itemize}
\item[(III)]
There is a constant $C_2$ such that the following holds. 
For all $a_1,a_2\in\I$, $a_1\le a_2$, and $j\ge1$ there is a mapping
$$
\S_{a_1,a_2,j}:\P_j(a_1)\to\P_j(a_2),
$$
such that, for all $\om\in\P_j(a_1)$, the elements $\om$ and $\S_{a_1,a_2,j}(\om)$ have the same symbolic dynamics:
\begin{equation}
\label{eq.drei1}
\ind_{a_1}(T_{a_1}^i(\om))=\ind_{a_2}(T_{a_2}^i(\S_{a_1,a_2,j}(\om))),\quad0\le i<j, 
\end{equation}
their images lie close together:
\begin{equation}
\label{eq.dreidist}
\dist(T_{a_1}^j(\om),T_{a_2}^j(\S_{a_1,a_2,j}(\om)))\le C_2|a_1-a_2|,
\end{equation}
and the size of the image of $\om$ can be estimated above by the size of the image of $\S_{a_1,a_2,j}(\om)$:
\begin{equation}
\label{eq.drei2}
|T_{a_1}^j(\om)|\le C_2|T_{a_2}^j(\S_{a_1,a_2,j}(\om))|.
\end{equation}
\end{itemize}

Finally, we state the main result of this paper.

\begin{Thm}
\label{t.main}
Let $T_a:[0,1]\to[0,1]$, $a\in\I$, be a piecewise expanding one-parameter family as described in 
Section~\ref{ss.prel}, satisfying properties~(i)-(iii), and conditions~(II) and (III). If for a piecewise $C^1$ map $X:\I\to[0,1]$ 
condition~(I) is fulfilled, then $X(a)$ is typical for $\mu_a$ for Lebesgue almost every $a\in\I$.
\end{Thm}

In the considered examples below, we will usually not verify conditions (I)-(III)
for the whole interval $\I$ for which the corresponding 
family is defined. Instead we will cover $\I$ by a countable number of smaller intervals 
and verify these conditions on these smaller intervals.

\subsection{Proof of Lemma~\ref{l.startcalc}}
\label{ss.lem22}
Let $j\ge1$ and assume in the following formulas that, for the parameter values $a\in\I$ under 
consideration, $y_j$ and $T_a^j$ are differentiable in $a$ and $Y(a)$, respectively.
For $0\le k<j$ we have 
\begin{equation}
\label{eq.startcalc1}
y_j'(a)=T_a^{j-k}\,'(y_k(a))y_k'(a)+\sum_{i=k+1}^jT_a^{j-i}\,'(y_i(a))\partial_aT_a(y_{i-1}(a)),
\end{equation}
which implies
\begin{equation}
\label{eq.startcalc}
\frac{y_j'(a)}{T_a^j\,'(Y(a))}
=\frac{1}{T_a^k\,'(Y(a))}\left(y_k'(a)+\sum_{i=k+1}^j\frac{\partial_aT_a(y_{i-1}(a))}{T_a^{i-k}\,'(y_k(a))}\right).
\end{equation}
For $j>j_0$, choosing $k=0$ and $k=j_0$, respectively, we get the following upper and lower bounds: 
\begin{equation}
\label{eq.startcalc2}
\frac{2L}{|T_a^{j_0}\,'(Y(a))|}\le\left|\frac{y_j'(a)}{T_a^j\,'(Y(a))}\right|\le
\sup_{a\in\I}\left(|Y'(a)|+\frac{\sup_{x\in K(a)}|\partial_aT_a(x)|}{\la-1}\right),
\end{equation}
where for the lower bound we used the assumption \eref{eq.startcomp}. Setting $X(a)=y_{j_0}(a)$, 
this implies the existence of a constant $C_0$ as required in condition~(I). 
It is only left to show that for each $j\ge 1$ the number of $a\in\I$ which 
are not contained in any element $\om\in\P_j|\I$ is finite (the partition $\P_j|\I$ is taken w.r.t. the map $X(a)$). 
This is easily done by induction over $j$. 
By the assumption in Lemma~\ref{l.startcalc}, $y_{j_0}:\I\to[0,1]$ is not differentiable only in a finite number of points 
and, further, $y_{j_0}(a)\in K(a)$. Since $y_{j_0}'(a)>L$ and since by property (i) the boundary points $b_k(a)$ are $\Lip(L)$, 
we have that $y_{j_0}(a)\in K(a)\setminus\{b_1(a),...,b_k(a)\}$ for all but finitely many $a\in\I$.
Thus, the number of $a\in\I$ which are not contained in any element $\om\in\P_1|\I$ is finite. 
Assume that $j\ge1$ and consider the partition $\P_{j+1}|\I$. From the lower bound in \eref{eq.startcalc2}, we derive that 
$|x_j'(a)|=|y_{j_0+j}'(a)|\ge\la^{j-j_0}2L>L$ for all $a$ contained in an element of $\P_j|\I$. 
As above we derive that only a finite number of $a$ contained in an element of $\P_j|\I$ can be mapped by $x_{j+1}$ to a 
partition point $b_k(a)$, $1\le k\le p_0$. This concludes the proof of Lemma~\ref{l.startcalc}.

\section{Proof of Theorem~\ref{t.main}}
\label{s.proof}

The idea of the proof of Theorem~\ref{t.main} is inspired by Chapter III in 
\cite{bc} where Benedicks and Carleson prove the existence of an a.c.i.p. for a.e. parameter in a certain 
parameter set (the set $\Delta_{\infty}$). Their argument implies that the critical point is in fact typical 
for this a.c.i.p. 

Let 
$$
\mathcal{B}:=\left\{(q-r,q+r)\cap[0,1]\ ;\ q\in\rational, r\in\rational^+\right\}.
$$
We will show that there is a constant $C\ge1$ such that for each $B\in\mathcal{B}$ the function 
$$
F_n(a)=\frac{1}{n}\sum_{j=1}^{n}\chi_B(x_j(a)),
\quad n\geq1,
$$
fulfills
\begin{equation}
\label{eq.limsup}
\varlimsup_{n\to\infty} F_n(a)\leq C|B|,\qquad\text{for a.e. }a\in\I.
\end{equation}
By standard measure theory (see, e.g., \cite{m}), ~\eref{eq.limsup} implies that, for a.e. $a\in\I$, 
every weak-$*$ limit point $\nu_a$ of
\begin{equation}
\label{eq.weakstarlimit}
\frac{1}{n}\sum_{j=1}^n\delta_{x_j(a)},
\end{equation}
has a density which is bounded above by $C$. In particular, $\nu_a$ is absolutely continuous. 
Observe that, by the definition of $x_j(a)$, the measure $\nu_a$ is also invariant for $T_a$ and, hence,  
$\nu_a$ is an a.c.i.p. for $T_a$. By the uniqueness of the a.c.i.p. for $T_a$, we finally derive that, 
for a.e. $a\in\I$, the weak-$*$ limit of \eref{eq.weakstarlimit} exists and
coincides with the a.c.i.p. $\mu_a$. This concludes the proof of Theorem~\ref{t.main}.

In order to prove ~\eref{eq.limsup}, it is sufficient to show that 
for all (large) integers $h\geq1$ there is an integer $n_{h,B}$, growing for fixed $B$ at most exponentially in $h$, 
such that
$$
\int_{\I}F_n(a)^h da\le\const(C|B|)^h,
$$ for all $n\ge n_{h,B}$ (see Lemma~A.1 in \cite{bs}). 

In the remaining part of this section, we assume that $B\in\mathcal{B}$ is fixed. For $h\geq1$, we have
\begin{equation}
\label{eq.integral}
\int_{\I}F_n(a)^h da=\sum_{1\leq j_1,...,j_h\leq n}\frac{1}{n^h}\int_{\I}\chi_B(x_{j_1}(a))
\cdot\cdot\cdot\chi_B(x_{j_h}(a))da.
\end{equation}
Observe that for a fixed parameter $a$, there exist an integer $k$ and a set $A\subset K(a)$ (which is the union of finitely many intervals) 
such that the system $T_a^k:A\to A$ with the measure $\mu_a$ is exact and, hence, it  
is mixing of all degrees (see \cite{rohlin} and \cite{wagner}). It follows that for sequences 
of non-negative integers $j_1^r,...,j_h^r$, $r\geq1$, with 
$$
\lim_{r\to\infty}\inf_{i\neq l}|j_i^r-j_l^r|=\infty,
$$
one has 
\begin{multline}
\label{m.exact}
\int_{A}\chi_B\Big(T_a^{kj_1^r}(x)\Big)\cdot\cdot\cdot\chi_B\left(T_a^{kj_h^r}(x)\right)d\mu_a(x)\\
=\mu_a\left(T_a^{-kj_1^r}(B)\cap...\cap T_a^{-kj_h^r}(B)\cap A\right)
\overset{r\to\infty}{\longrightarrow}\mu_a(B\cap A)^h\leq\left(\|\varphi_a\|_\infty|B|\right)^h.
\end{multline}
Since the maps $T_a^j$ and $x_j$ are, by conditions (I)-(III), 'comparable' it is natural to expect 
similar mixing properties for the maps $x_j$. In fact, in the Section~\ref{ss.proofpropmain}, we are going to prove 
the following statement.

\begin{Prop}
\label{p.main}
Assume that conditions (I)-(III) are satisfied. 
Disregarding a finite number of parameter values in $\I$, we can cover 
$\I$ by a countable number of intervals $\tilde{\I}\subset\I$ such that for each such interval $\tilde{\I}$ 
there is a constant $C\ge1$ such that the following holds. 
For all $h\geq1$, there is an integer $n_{h,B}$ growing at most exponentially in $h$ such that, for all 
$n\ge n_{h,B}$ and for all integer $h$-tuples $(j_1,...,j_h)$ with $\sqrt{n}\le j_1<j_2<...<j_h\leq n-\sqrt{n}$ and 
$j_l-j_{l-1}\ge\sqrt{n}$, $l=2,...,h$, 
$$
\int_{\tilde{\I}}\chi_B(x_{j_1}(a))\cdot\cdot\cdot
\chi_B(x_{j_h}(a))da\le(C|B|)^h.
$$
\end{Prop}

Seen from a more probabilistic point of view, Proposition~\ref{p.main} says that whenever the distances between 
consecutive $j_i$'s are sufficiently large, the behavior of the $\chi_B(x_{j_i}(.))$'s is comparable to that 
of independent random variables.
Now, the number of $h$-tuples $(j_1,...,j_h)$ in the sum in ~\eref{eq.integral}, for which $\min_ij_i<\sqrt{n}$ or 
$\min_{k\neq l}|j_{k}-j_{l}|<\sqrt{n}$, is bounded by $2h^2n^{h-1/2}$. Hence, by Proposition~\ref{p.main}, we obtain
$$
\int_{\tilde{\I}}F_n(a)^h da\le(C|B|)^h+\frac{2h^2}{\sqrt{n}}|\tilde{\I}|\le2(C|B|)^h,
$$
whenever 
$$
n\geq\max\left\{n_{h,B},
\left(\frac{2h^2|\tilde{\I}|}{(C|B|)^h}\right)^2\right\}.
$$ 
Since both terms in this lower bound for $n$ grow at most exponentially in $h$, this implies \eref{eq.limsup}, 
for a.e. $a\in\tilde{\I}$. This concludes the proof of Theorem~\ref{t.main}.

\subsection{Proof of Proposition~\ref{p.main}}
\label{ss.proofpropmain}
First we cover the parameter interval $\I$ by smaller intervals $\tilde{\I}$ which will be more convenient 
in the proof of Lemma~\ref{l.remainderintervals} below. Fix an integer $t_0$ so large that $2^{1/t_0}\le\sqrt{\la}$. 
For $a\in\I$, let
$$
\delta(a)=\min\{|\om|\ ;\ \om\in\P_{t_0}(a)\}>0.
$$
Since condition~(III) holds, we can argue as in the proof of Lemma~\ref{l.app1} (see inequality~\eref{eq.kaap}) and, 
disregarding a finite number of parameter values in $\I$, we can cover 
$\I$ by a countable number of closed intervals $\tilde{\I}\subset\I$ such that for each such interval $\tilde{\I}$ there is a constant 
$\delta=\delta(\tilde{\I})>0$ such that  
\begin{equation}
\label{eq.kaap1}
\delta(a)\ge\delta,
\end{equation} 
for all $a\in\tilde{\I}$. We assume also that $\tilde{\I}$ is chosen such that $X:\tilde{\I}\to[0,1]$ is 
$C^1$ (and not piecewise $C^1$).
Henceforth, we fix such an interval $\tilde{\I}$. Instead of $\tilde{\I}$ we will write again $\I$.

Conditions (I)-(III) enables us to switch from maps on the parameter space to maps on the dynamical interval. 
In order to have good distortion estimates, we split up the integral in Proposition~\ref{p.main} and 
integrate over smaller intervals of length $1/n$. 
In the following, we are going to show that there exist a constant $C\ge1$ and an 
integer $n_{h,B}$ growing at most exponentially in $h$ such that, for $n\ge n_{h,B}$, we have
\begin{equation}
\label{eq.JJJ}
\int_{J}\chi_B(x_{j_1}(a))\cdot\cdot\cdot\chi_B(x_{j_h}(a))da\le|J|(C|B|)^h,
\end{equation}
for all intervals $J\subset\I$ of length $1/n$ and all $h$-tuples $(j_1,...,j_h)$ as described in Proposition~\ref{p.main}.
If $n_{h,B}\gg |I|^{-1}$, this immediately implies that, for $n\ge n_{h,B}$,
$$
\int_{\I}\chi_B(x_{j_1}(a))\cdot\cdot\cdot\chi_B(x_{j_h}(a))da\le|\I|(C|B|)^h,
$$
which concludes the proof of Proposition~\ref{p.main} (where one has to adapt the constant $C$ in 
the statement of Proposition~\ref{p.main} to $\max\{|\I|C,C\}$).

Let $n$ be large and $J\subset\I$ an open interval of size $1/n$.
Note that by condition~(I), for $j\ge1$, there are only finitely many parameter values not contained in any element of $\P_j|J$. Hence, 
we can neglect such parameter values and focus on the partitions $\P_j|J$. Consider the set 
$$
\Omega_J=\{\om\in\P_n|J\ ;\ x_{j_i}(\om)\cap B\neq\emptyset,\ 1\le i\le h\}.
$$
Obviously if we show that $|\Omega_J|\le|J|(C|B|)^h$, this implies \eref{eq.JJJ}. 
Our strategy of establishing this estimate for $\Omega_J$ is to compare the elements in $\Omega_J$ with elements in the partition 
$\P_n(\a)$ where $\a$ denotes the right boundary point of $J$. 
The following lemma will allow us to switch from partitions on the parameter space to partitions on the phase space. 
Its proof is given in Section~\ref{ss.distortion1}.

\begin{Lem}
\label{l.distortion1}
Under the assumption that conditions (I) and (III) are satisfied, there is an integer $q\ge1$ and a constant $C'\ge1$ such that 
for all open intervals $J\subset\I$ of length $1/n$ the following holds. There is an at most $q$-to-one map
$$
\S_J:\P_n|J\to\P_n(\a),
$$
such that, for $\om\in\P_n|J$, the images of $\om$ and $\S_J(\om)$ are close:
\begin{equation}
\label{eq.sun2}
\dist(x_j(\om),T_{\a}^j(\S_J(\om)))\le C'/n,\quad\text{for}\ 0\le j\le n-\sqrt{n},
\end{equation}
and the size of $\om$ is controlled by the size of $\S_J(\om)$:
\begin{equation}
\label{eq.sun}
|\om|\le C'|\S_J(\om)|.
\end{equation}
\end{Lem}

Let $J_x\subset[0,1]$ be the interval obtained by the interval having the same midpoint as $X(J)$ and having size 
$|X(J)|+3C'n^{-1}$ intersected with the support $K(\a)$ of the a.c.i.p. for $T_{\a}$. 
(Observe that $\I$ is chosen such that $X:\I\to[0,1]$ has no discontinuities.) 
Lemma~\ref{l.distortion1} allows us to switch from considering $\Omega_J$ to considering the set 
$$
\Omega=\{\om\in\P_n(\a)|J_x\ ;\ T_{\a}^{j_i}(\om)\cap2B\neq\emptyset,\ 1\le i\le h\},
$$
where $2B$ denotes the interval twice as long as $B$ and having the same midpoint as $B$. 
Indeed if $n\gg|B|$ then, by \eref{eq.sun2}, for all $\om\in\Omega_J$, there exists $\om'\in\Omega$ 
such that $\om'=\S_J(\om)$. Thus,  by \eref{eq.sun}, and since the map $\S_J:\Omega_J\to\Omega$ is at most $q$-to-one we obtain
$|\Omega_J|\le qC'|\Omega|$.
Fix an integer $\ta\ge1$ such that $\la^{-\ta}\le|B|/2$, and assume that $n$ is 
so large that $\sqrt{n}\ge\ta$, which ensures that there are at least $\ta$ iterations between two
consecutive $j_i$'s (and between $j_h$ and $n$).
Let $\Omega_0=J_x$ and, for $1\le i\le h$, set 
$$
\Omega_i=\{\om\in\P_{j_i+\ta}(\a)|\Omega_{i-1}\ ;\ T_{\a}^{j_i}(\om)\cap2B\neq\emptyset\}.
$$
Clearly, $\Omega\subset\Omega_h$. Notice that, by the assumption on $\tau$, we have $|T_{\a}^{j_i}(\om)|\le|B|/2$ for 
all $\om\in\P_{j_i+\tau}(\a)$ and, thus,
$$
\Omega_i\subset\{x\in\Omega_{i-1}\ ;\ T_{\a}^{j_i}(x)\in3B\}.
$$ 

Note that, since the density $\varphi_{\a}$ is a fixed point of the Perron-Frobenius operator, we have, for $k\ge1$, 
$$
\varphi_{\a}(y)
            =\sum_{\begin{subarray}{c}
              x\in K(\a)\\
              T_{\a}^k(x)=y
             \end{subarray}}
  \frac{\varphi_{\a}(x)}{|T_{\a}^k\,'(x)|},
$$
for a.e. $y\in K(\a)$. By condition~(II), we get
\begin{equation}
\label{eq.perronfrobenius}
\sum_{\begin{subarray}{c}
              x\in K(\a)\\
              T_{\a}^k(x)=y
            \end{subarray}}
  \frac{1}{|T_{\a}^k\,'(x)|}\le C_1^2,
\end{equation}
for a.e. $y\in K(\a)$. For $\ome\in\P_j(\a)$, $j\ge1$, it follows that 
\begin{align*}
|T_{\a}^j(\{x\in\ome\ ;\ T_{\a}^{j+k}(x)\in3B\})|&=\int_{3B}\sum_{\begin{subarray}{c}
     x\in\ome \\
     T_{\a}^{j+k}(x)=y
     \end{subarray}}
     \left|\frac{T_{\a}^j\,'(x)}{T_{\a}^{j+k}\,'(x)}\right|dy\\
     &\le \int_{3B}\sum_{\begin{subarray}{c}
              x\in K(\a)\\
              T_{\a}^k(x)=y
            \end{subarray}}
  \frac{1}{|T_{\a}^k\,'(x)|}dy\le3C_1^2|B|.
\end{align*}
By Lemma~\ref{l.distortion} (where we set $a_1=a_2=\a$), which is stated and proven in 
the next section, we get the distortion estimate
$|T_{\a}^{j}\,'(x)|/|T_{\a}^{j}\,'(x')|\le C_3$, for $x,x'\in\ome$.
Now if $|T_{\a}^j(\ome)|\ge\delta$ then we get
\begin{align}
\label{eq.anteil}
\nonumber
|\{x\in\ome\ ;\ T_{\a}^{j+k}(x)\in3B\}|
 &\le C_3\frac{|T_{\a}^{j}(\{x\in\ome\ ;\ T_{\a}^{j+k}(x)\in3B\})|}{|T_{\a}^{j}(\ome)|}|\ome|\\ 
 &\le\frac{3C_1^2C_3}{\delta}|B||\ome|.
\end{align}

In the remaining part of this section let $\tilde{C}=3C_1^2C_3\delta^{-1}$. 
In order to apply \eref{eq.anteil}, we will exclude in each set $\Omega_i$ certain intervals with too short images. 
To this end we define, for $0\le i\le h-1$, 
the following exceptional sets (let $j_0=0$):
\begin{multline*}
E_i=\{\om\in\P_{j_{i+1}}(\a)|\Omega_i\ ;\ \nexists\;\ome\in\P_{j_i+k}(\a)|\Omega_i,\ \ta\le k\le j_{i+1}-j_i,\\ 
\text{s.t.}\ \ome\supset\om\ \text{and }|T_{\a}^{j_i+k}(\ome)|\ge\delta\}.
\end{multline*}
The following lemma gives an estimate on the size of these exceptional sets. It is proven in Section~\ref{ss.remainders}.

\begin{Lem}
\label{l.remainderintervals}
There is a number $n_{h,B}$ growing at most exponentially in $h$ such that
$$
|E_i|\le\frac{(\tilde{C}|B|)^h|\Omega_0|}{h},
$$
for all $0\le i\le h-1$ and $n\ge n_{h,B}$.
\end{Lem}

Disregarding finitely many points, 
$\Omega_i\setminus E_i$, $0\le i\le h-1$, can be seen as a set of disjoint and open intervals $\ome$ such that each $\ome$ is an 
element of a partition $\P_{j_i+k}|\Omega_i$, $\ta\le k\le j_{i+1}-j_i$, and 
$|T_{\a}^{j_i+k}(\ome)|\ge\delta$. By ~\eref{eq.anteil}, we obtain 
$$
|\{x\in\ome\ ;\ T_{\a}^{j_{i+1}}(x)\in3B\}|\le \tilde{C}|B||\ome|,
$$
which in turn implies that, for $n\ge n_{h,B}$, 
\begin{align*}
|\Omega_{i+1}|
               \le \tilde{C}|B||\Omega_i\setminus E_i|+|E_i|
               \le \tilde{C}|B||\Omega_i|+\frac{(\tilde{C}|B|)^h|\Omega_0|}{h}.
\end{align*}
Hence, we have
\begin{align*}
|\Omega|\le|\Omega_h|\le (\tilde{C}|B|)^h|\Omega_0|+h\frac{(\tilde{C}|B|)^h|\Omega_0|}{h}\le2(\tilde{C}|B|)^h|\Omega_0|.
\end{align*}
Observe that $|\Omega_0|=|J_x|\le(3C'+\sup_{a\in\I}|X'(a)|)|J|$. Since $|\Omega_J|\le qC'|\Omega|$ 
we conclude that $|\Omega_J|\le(C|B|)^h|J|$ where
$C=2q\tilde{C}C'(3C'+\sup_{a\in\I}|X'(a)|)$. This implies \eref{eq.JJJ} which is the estimate 
we had to show.

\section{Switching from the parameter space to the phase space, and estimating the set of partition elements with too small images}
\label{s.keylemmas}
In this section we will prove the key lemmas, Lemma~\ref{l.distortion1} and Lemma~\ref{l.remainderintervals}, in the 
proof of Proposition~\ref{p.main}. As seen in Section~\ref{ss.proofpropmain}, Lemma~\ref{l.distortion1} makes it possible 
to compare partition elements on the parameter space to partition elements on the phase space, and 
Lemma~\ref{l.remainderintervals} provides us with a good estimate of exceptional partition elements in the phase 
space with too small images.
We establish first a distortion lemma. 

\begin{Lem}
\label{l.distortion}
If the one-parameter family $T_a$, $a\in\I$, satisfies condition~(III), then there exists a constant $C_3\ge1$ 
such that we have the following distortion estimate.  Let $n\ge1$ and $a_1,a_2\in\I$ such that 
$a_1\le a_2$ and $a_2-a_1\le1/n$. For $\om\in\P_j(a_1)$, $1\le j\le n$, we have 
$$
\frac{1}{C_3}\le\left|\frac{T_{a_1}^j\,'(x)}{T_{a_2}^j\,'(x')}\right|\le C_3,
$$
for all $x\in\om$ and $x'\in\S_{a_1,a_2,j}(\om)$.
\end{Lem}

\begin{Rem}
If $a_1=a_2$ in Lemma~\ref{l.distortion}, then we get a standard distortion estimate for piecewise expanding 
$C^{1,1}(L)$ maps.
\end{Rem}

\begin{proof}
Fix $\tau\ge1$ such that $2L/\tau<\delta_0$. 
Taking constant $C_3$ in Lemma~\ref{l.distortion} greater than $(\Lambda/\la)^\tau$, for $n\le\tau$, 
the distortion estimate is trivially satisfied and we can assume that 
$\tau\le j\le n$. Observe that, by condition (III), there exist points $r_0\in T_{a_1}^j(\om)$ 
and $s_0\in T_{a_2}^j(\S_{a_1,a_2,j}(\om))$ such that $|r_0-s_0|\le 2C_2/n$.
For $1\le i\le j$, let
$$
r_i\in T_{a_1}^{j-i}(\om),\quad\text{and}\quad s_i\in T_{a_2}^{j-i}(\S_{a_1,a_2,j}(\om)),
$$
be the pre-images of $r_0$ and $s_0$, i.e., $T_{a_1}^i(r_i)=r_0$ and $T_{a_2}^i(s_i)=s_0$.  Note that, by (III), we have 
$\ind_{a_1}(r_i)=\ind_{a_2}(s_i)$. Let $k_i=\ind_{a_1}(r_i)$, and denote by $B_{k_i}(a)$ the (maximal) monotonicity 
interval $(b_{l-1}(a),b_l(a))$, $1\le l\le p_0$, for $T_a:[0,1]\to[0,1]$ which contains the domain $D_{k_i}(a)$. 
Recall that, by property~(i), $|B_{k_i}(a)|\ge\delta_0$. 

\begin{Claim}
The distance between $r_i$ and $s_i$, $1\le i\le j$, satisfies
\begin{equation}
\label{eq.distance}
|r_i-s_i|\le\frac{5L\Lambda+2C_2}{\la-1}\frac{1}{n}.
\end{equation}
\end{Claim}

\begin{proof}
In order to show ~\eref{eq.distance}, we will show inductively that
\begin{equation}
\label{eq.distance1}
|r_i-s_i|\le\frac{L(1+2\la+2\Lambda)}{n}\sum_{l=1}^{i}\frac{1}{\la^l}+\frac{2C_2}{n\la^i},
\end{equation}
for $0\le i\le j$. This is obviously true for $i=0$, so assume that \eref{eq.distance1} holds for $i-1$ where  $1\le i\le j$. 
Recall that $a_2-a_1\le1/n$. Hence, by property~(i), there is an interval $B\subset B_{k_i}(a_1)\cap B_{k_i}(a_2)$ such that 
$|B_{k_i}(a_1)\setminus B|\le2L/n$ and $B\subset B_{k_i}(a)$, for all $a\in[a_1,a_2]$. 
(Since $|B_{k_i}(a_1)|\ge\delta_0$ and $2L/n<\delta_0$ the interval $B$ is non-empty.)
Take a point $\tilde{r}_i\in B$ such that $|r_i-\tilde{r}_i|\le2L/n$. Since $|\tilde{r}_i-s_i|\le\la^{-1}|T_{a_2}(\tilde{r}_i)-T_{a_2}(s_i)|$, 
we obtain
$$
|r_i-s_i|\le\frac{2L}n+\frac1\la|T_{a_2}(\tilde{r}_i)-T_{a_1}(\tilde{r}_i)|+\frac1\la|T_{a_1}(\tilde{r}_i)-s_{i-1}|.
$$
By property (ii) the map $a\mapsto T_a(\tilde{r}_i)$ is $\Lip(L)$ on $[a_1,a_2]$ which implies 
$|T_{a_2}(\tilde{r}_i)-T_{a_1}(\tilde{r}_i)|\le L/n$. Further, we have 
$$
|T_{a_1}(\tilde{r}_i)-s_{i-1}|\le|T_{a_1}(\tilde{r}_i)-r_{i-1}|+|r_{i-1}-s_{i-1}|\le\frac{2L\Lambda}n+|r_{i-1}-s_{i-1}|.
$$
Altogether, we obtain
$$
|r_i-s_i|\le\frac{|r_{i-1}-s_{i-1}|}\la+\frac{L(1+2\la+2\Lambda)}{\la n},
$$
and we easily deduce that \eref{eq.distance1} holds for $i$.
\end{proof}

For $1\le i\le j$, let $B\subset B_{k_i}(a_1)\cap B_{k_i}(a_2)$ and $\tilde{r}_i\in B$ as in the proof of \eref{eq.distance}. 
We obtain
\begin{align*}
|T_{a_1}'(r_i)-T_{a_2}'(s_i)|&\le|T_{a_1}'(r_i)-T_{a_1}'(\tilde{r}_i)|+|T_{a_1}'(\tilde{r}_i)-T_{a_2}'(\tilde{r}_i)|+|T_{a_2}'(\tilde{r}_i)-T_{a_2}'(s_i)|\\
&\le L|r_i-\tilde{r}_i|+L|a_1-a_2|+L|\tilde{r}_i-s_i|\le\frac{4L^2}n+\frac{L}n+L|r_i-s_i|,
\end{align*}
where we used property~(ii) for estimating the term $|T_{a_1}'(\tilde{r}_i)-T_{a_2}'(\tilde{r}_i)|$.
Using \eref{eq.distance}, we get that $|T_{a_1}'(r_i)-T_{a_2}'(s_i)|\le C'n^{-1}$ where 
$C'=L(4L+1+(5L\Lambda+2C_2)(\la-1)^{-1})$.
Altogether, we obtain
\begin{align}
\label{eq.distortion11}
\nonumber
\left|\frac{T_{a_1}^j\,'(x)}{T_{a_2}^j\,'(x')}\right|&\le
\prod_{i=1}^j\frac{|T_{a_1}'(T_{a_1}^{j-i}(x))|}{|T_{a_2}'(T_{a_2}^{j-i}(x'))|}
\le\prod_{i=1}^j
\frac{|T_{a_1}'(r_i)|+L|T_{a_1}^{j-i}(\om)|}{\max\{|T_{a_2}'(s_i)|-L|T_{a_2}^{j-i}(\S_{a_1,a_2,j}(\om))|,\la\}}\\
&\le\prod_{i=1}^j
\frac{|T_{a_2}'(s_i)|+C'n^{-1}+L\la^{-i}}
{\max\{|T_{a_2}'(s_i)|-L\la^{-i},\la\}}.
\end{align}
Since $j\le n$, the product in the last term of inequality 
\eref{eq.distortion11} is clearly bounded above by a constant independent on $n$. Hence, 
this shows the upper bound in the distortion estimate. The lower bound is shown in 
the same way. 
\end{proof}

\subsection{Proof of Lemma~\ref{l.distortion1}}
\label{ss.distortion1}
We define the map
$$
\S_J:\P_n|J\to\P_n(\a)
$$
as follows. Let $\om\in\P_n|J$ and $a\in\om$. By the definition of the partitions associated to 
the parameter interval, we have $x_j(a)\notin\{b_0(a),...,b_{p_0}(a)\}$, for all $0\le j<n$ (recall that $X(a)=x_0(a)$).  
Hence, there exists an element 
$\om(X(a))$ in the partition $\P_n(a)$ containing the point $X(a)$. We set
$$
\S_J(\om)=\S_{a,\a,n}(\om(X(a))),
$$
where $\S_{a,\a,n}:\P_n(a)\to\P_n(\a)$ is the map given by (III). 
Note that the element $\om'=\S_J(\om(X(a)))$ has the same combinatorics as $\om$, i.e., 
$$
\ind_{\a}(T_{\a}^j(\om'))=\ind(x_j(\om)),
$$ 
$0\le j<n$. 
Since there cannot be two elements in $\P_n(\a)$ with the same combinatorics, the element $\om'$ is independent on the 
choice of $a\in\om$. It follows that the map $\S_J$ is well-defined. 

From the first claim in the proof of Lemma~\ref{l.app1}, we get that the boundary points of $T_a^n(\om(X(a)))$ change continuously in 
$a\in\om$. Hence, since $x_n(a)$ is contained in $T_a^n(\om(X(a)))$ and since $x_n'(a)\neq0$, for all $a\in\om$, we get 
\begin{multline*}
|x_n(\om)|\le\lim_{a\to\om_L}\lim_{a'\to\om_R}\Big(|T_a^n(\om(X(a)))|+\dist(T_a^n(\om(X(a))),T_{a'}^n(\om(X(a'))))\\
+|T_{a'}^n(\om(X(a')))|\Big),
\end{multline*}
where $\om_L$ and $\om_R$ denote the left and right endpoint of $\om$, respectively.
By \eref{eq.drei2}, we get $|T_a^n(\om(X(a)))|\le C_2|T_{\a}^n(\S_J(\om))|$, for all $a\in\om$, and, by \eref{eq.dreidist}, 
we obtain 
$$
\lim_{a\to\om_L}\lim_{a'\to\om_R}\dist(T_a^n(\om(X(a))),T_{a'}^n(\om(X(a'))))\le C_2|\om|.
$$
From condition~(I), it follows that $|\om|\le C_0\la^{-n}|x_n(\om)|$. Thus, we deduce that
$$
|x_n(\om)|\le\frac{2C_2}{1-C_0C_2\la^{-n}}|T_{\a}^n(\S_J(\om))|,
$$
where in the last inequality we used once more \eref{eq.drei2}. By condition~(I) and Lemma~\ref{l.distortion}, we obtain that 
$|T_{\a}^n\,'(x)|\le C_0C_3|x_n'(a)|$, for all $a\in\om$ and $x\in\S_J(\om)$. We conclude that
$$
|\om|\le\frac{C_0C_2C_3(1+C_2)}{1-C_0C_2\la^{-n}}|\S_J(\om)|,
$$
which implies the estimate \eref{eq.sun} in Lemma~\ref{l.distortion1}.

In order to prove \eref{eq.sun2}, observe that $|a_J-a|\le|J|\le1/n$, for all $a\in\om$. Hence, by \eref{eq.distance}, we have
$$
\dist(T_a^j(\om(X(a))),T_{\a}^j(\S_J(\om)))\le\frac Cn,\quad\text{for all }0\le j\le n,
$$
where $C$ is the constant in the righthand side of \eref{eq.distance}. For $n$ sufficiently large 
we have $|T_a^j(\om(X(a)))|\le\la^{-(n-j)}\le Cn^{-1}$, if $j\le n-\sqrt{n}$. Thus, since $x_j(a)$ is contained in $T_a^j(\om(X(a)))$, we conclude
$$
\dist(x_j(\om),T_{\a}^j(\S_J(\om)))\le\frac{2C}n,\quad\text{for all }0\le j\le n-\sqrt{n}.
$$

In order to conclude the proof of 
Lemma~\ref{l.distortion1}, it is only left to show that the map $\S_J$ is at most $q$-to-one for some integer $q\ge1$. 
Let $l_0=l_0(C_0,\la)\ge0$ be so large that $|x_j'(a)|\ge L$ for all $j\ge l_0$ and parameter values $a\in\I$ for which 
the derivative is defined ($L$ is the 
Lipschitz constant introduced in Section~\ref{ss.prel}). If $\ome\in\P_{l_0}|J$, using that the partition points 
$b_0(a),...,b_{p_0}(a)$ are $\Lip(L)$ on $\I$, 
it is easy to show that the map $\S_J|_{\ome}:\P_n|\ome\to\P_n(\a)$ is one-to-one. Hence, setting $q=\#\{\om\in\P_{l_0}|\I\}$ 
we derive that the map $\S_J:\P_n|J\to\P_n(\a)$ is at most $q$-to-one.

\subsection{Proof of Lemma~\ref{l.remainderintervals}}
\label{ss.remainders}
Let $j\ge1$ and $a\in\I$. For each $\om'\in\P_j(a)$, we define the set
\begin{multline*}
E_{\om'}=\{\om\in\P_{j+[\sqrt{n}]}(a)|\om'\ ;\ \nexists\;\tilde{\om}\in\P_{j+k}(a)|\om',\\ 
   \tau\le k\le[\sqrt{n}],\ \text{s.t.}\ \tilde{\om}\supset\om\ \text{and}\ |T_a^{j+k}(\tilde{\om})|\ge\delta\}.
\end{multline*}
Observe that the choice of $\delta$ in the beginning of Section~\ref{ss.proofpropmain} 
implies that if $\om\in\P_t(a)$, $1\le t\le t_0$, then $|T_a^t(\om)|\ge\delta$.
From this we deduce that if $\hat{\om}\in\P_l(a)$, $l\ge1$, and $1\le t\le t_0$ then we have
$$
\#\{\om\in\P_{l+t}(a)|\hat{\om}\ ;\ |T_a^{l+t}(\om)|<\delta\}\le2.
$$
In other words only the elements in $\P_{l+t}(a)|\hat{\om}$ that are adjacent to a boundary point of $\hat{\om}$ can have a 
small image. By a repeated use of this fact and using that $\#\{\om\in\P_{j+\tau}(a)|\om'\}\le p_1^\tau$ 
(recall that $p_1$ is the number of elements in $\P_1(a)$ and, by property~(iii), $p_1$ does not depend on $a$), we derive 
$$
\#\{\om\in\P_{j+[\sqrt{n}]}(a)|E_{\om'}\}\le p_1^\tau\cdot2\cdot2^{([\sqrt{n}]-\tau)/t_0}\le2p_1^\tau\sqrt{\la}^{[\sqrt{n}]},
$$
where in the last inequality we used the definition of $t_0$. It follows that
\begin{equation}
\label{eq.gagaga}
|T_a^j(E_{\om'})|\le \frac{\#\{\om\in\P_{j+[\sqrt{n}]}(a)|E_{\om'}\}}{\la^{[\sqrt{n}]}}
               \le\frac{2p_1^\tau}{\sqrt{\la}^{[\sqrt{n}]}}=:\gamma_n.
\end{equation}
Observe that if we choose $j=j_i$, $0\le i\le h-1$, and $a=\a$ then the exceptional set $E_i$ in Lemma~\ref{l.remainderintervals} is contained in
$$
E=\bigcup_{\om'\in\P_j(a)}E_{\om'}.
$$
We have
$$
|E|=\sum_{\om'\in\P_j(a)}\int_{T_a^j(E_{\om'})}\frac{1}{|T_a^j\,'(x_y)|}dy,
$$
where $x_y=(T_a^j|_{\om'})^{-1}(y)$. Set
$$
\O_j(a)=\{b\ ;\ b\in\partial T_a^i(\om),\ 1\le i\le j,\ \om\in\P_j(a)\},
$$
and, for $\om\in\P_j(a)$, let $\Gamma(\om)=[b_{\om},b_{\om}+\gamma_n]\cap T_a^j(\om)$, where 
$b_{\om}\in\O_j(a)$ denotes the left boundary point of $T_a^j(\om)$. By \eref{eq.gagaga} and 
the distortion estimate in Lemma~\ref{l.distortion} (where $a_1=a_2=a$), we get
$$
|E|\le C_3\sum_{\om'\in\P_j(a)}\int_{\Gamma(\om')}\frac{1}{|T_a^j\,'(x_y)|}dy,
$$
Now, summing over all points in $\O_j(a)$ and moving the sum over the partition elements inside the integral, we derive that
$$
|E|\le C_3\sum_{b\in\O_j(a)}\int_{[b,b+\gamma_n]}\sum_{\begin{subarray}{c}
           x\in K(a) \\
           T_a^j(x)=y
           \end{subarray}}
           \frac{1}{|T_a^j\,'(x)|}dy\le C_3C_1^2\gamma_n\#\O_j(a),
$$
where in the last inequality we used \eref{eq.perronfrobenius}.
Observe that for each $b\in\O_j(a)$ there is a monotonicity 
domain $D\in\P_1(a)$ for $T_a|_{K(a)}$ and a partition point $c\in\partial D$ such that 
$$
b=\lim_{\begin{subarray}{c} x\to c\\ x\in D\end{subarray}}T_a^i(x),
$$
for some $1\le i\le j$. Thus, since $j\le n$, we have $\#\O_j(a)\le\#\O_n(a)\le n\cdot2p_1$. 
Recall that $\tilde{C}=3C_1^2C_3\delta^{-1}$ and $|\Omega_0|\ge n^{-1}$. 
Finally, in the case when $j=j_i$, $0\le i\le h-1$, and $a=\a$ we deduce that
$$
|E_i|\le|E|\le 2p_1C_3C_1^2n\gamma_n\le \frac{(\tilde{C}|B|)^h|\Omega_0|}{h},
$$
for $n\ge n_{h,B}$, where $n_{h,B}$ can obviously be chosen to grow only exponentially in $h$. 
This concludes the proof of Lemma~\ref{l.remainderintervals}.

\section{$\beta$-transformation}
\label{s.beta}
We apply Theorem~\ref{t.main} to a $C^{1,1}(L)$-version of
$\beta$-transformations. Let the map $T:[0,\infty)\to[0,1]$ be piecewise $C^{1,1}(L)$ 
and $0=b_0<b_1<...$ be the associated partition, where $b_k\to\infty$ as $k\to\infty$. We assume that:
\begin{itemize}
\item[a)]
$T$ is right continuous and $T(b_k)=0$, for each $k\ge0$.
\item[b)]
For each $a>1$,
$$
1<\inf_{x\in[0,1]}\partial_xT(a x)\quad\text{and}\quad\sup_{x\in[0,1]}\partial_xT(ax)<\infty.
$$
\end{itemize}
See Figure~\ref{f.beta}. 
We define the one-parameter family $T_a:[0,1]\to[0,1]$, $a>1$, by 
$T_a(x)=T(ax)$. 
There exists a unique a.c.i.p. $\mu_a$ for each $T_a$ as the following lemma asserts.

\begin{Lem}
\label{l.uniqueb}
For each $a>1$ there exists a unique a.c.i.p. $\mu_a$ for $T_a$. The support $K(a)$ is an interval adjacent to $0$ 
and the map $a\mapsto|K(a)|$, $a>1$, is piecewise constant where the set of discontinuity points is countable and nowhere dense. 
\end{Lem}

The proof of Lemma~\ref{l.uniqueb} is given in Section~\ref{ss.eins1}. 
Henceforth, $\I\subsetneq(1,\infty)$ will always denote a closed interval 
on which $|K(a)|$ is constant as well as the number of discontinuities of $T_a$ inside $K(a)$ and $[0,1]$, 
i.e., the numbers $\#\{k\ge0\ ;\ b_k/a\in\inte(K(a))\}$ and $\#\{k\ge0\ ;\ b_k/a<1\}$ are constant on $\I$. For such an interval $\I$ it is now 
straightforward to check that the one-parameter family $T_a$, $a\in\I$, fits into the model described 
in Section~\ref{ss.prel} fulfilling properties (i)-(iii).
Now, we can state the main result of this section.

\begin{Thm}
\label{t.beta}
If for a $C^1$ map $X:\I\to[0,1]$ condition (I) is satisfied, then $X(a)$
is typical for $\mu_a$ for a.e. $a\in\I$.
\end{Thm}

\begin{Rem}
As the family $T_a$ we could also consider other models as, e.g., $x\mapsto ag(x)\mod1$ where $g:[0,1]\to[0,1]$ is a 
$C^{1,1}(L)$ homeomorphism with a strict positive derivative. Even if this model is not included in the families 
described above, it would be easier to treat since, seen as a map from the circle into itself, it is 
non-continuous only in the point $0$ which, in particular, implies that $K(a)=[0,1]$.
\end{Rem}

By Theorem~\ref{t.main}, in order to proof Theorem~\ref{t.beta}, 
it is sufficient to check conditions (II) and (III). We will show that there is a large class of 
maps $Y$ for which we have almost sure typicality:

\begin{Cor}
\label{c.beta}
If $Y:(1,\infty)\to(0,1]$ is $C^1$ such that $Y'(a)\ge0$, then $Y(a)$
is typical for $\mu_a$ for a.e. $a>1$.
\end{Cor}

\begin{Rem}
Observe that the map
$$
Y(a)\equiv\lim_{x\to b_k-}T(x),
$$
$a>1$, satisfies $Y(a)>0$ and $Y'(a)\ge0$, and, hence, Corollary~\ref{c.beta} can be applied
to these values which are important from a dynamical point of view. 
\end{Rem}

\subsection{Proof of Corollary~\ref{c.beta} and Lemma~\ref{l.uniqueb}}
\label{ss.eins1}
We prove first Lemma~\ref{l.uniqueb}.
For $a>1$ let $\mu_a$ be an a.c.i.p. for $T_a$ with support $K(a)$ and let $J\subset K(a)$ be an interval. 
Since $T_a$ is expanding there exists an integer $j\ge1$ such that $T_a^j:J\to[0,1]$ is not any longer continuous. 
It follows that $T_a^j(J)$ contains a neighborhood of $0$. If $T_a$ had more than one 
a.c.i.p. then, by \cite{wong}, there would exist two a.c.i.p.'s with disjoint supports (disregarding a finite number of 
points). This shows that the a.c.i.p. $\mu_a$ is unique and its support $K(a)$ contains an interval adjacent to $0$. 
Since the image by $T_a$ of an interval adjacent to $0$ is again an interval adjacent to $0$, using the ergodicity of 
$\mu_a$, one deduces that $K(a)$ is a single interval and $K(a)=\closure\{\cup_{j\ge0}T_a^j(L)\}$ where $L$ is 
a sufficiently small interval adjacent to $0$. Observe that for an arbitrary interval $L\subset[0,1]$ adjacent to $0$ 
one has $T_a(L)\subset T_{a'}(L)$ for all $a'>a$, which implies that the right end point $b(a)$ of $K(a)$ is non-decreasing. 
Now, we easily derive that there is $k\ge1$ such that $b_k/a<b(a)$ and 
$T_a(b_k/a-)=b(a)$. Since $T_a(b_k/a-)=T(b_k-)$ is constant in $a$, it follows that if $T_a(b(a)-)<b(a)$ then 
the support of the a.c.i.p. for $T_{a'}$ is equal to $K(a)$ for all $a'\in[a,\tilde{a}]$ where $a<\tilde{a}\le\infty$ 
is maximal such that $T_{a'}(b(a)-)\le b(a)$ for all $a'\in[a,\tilde{a}]$.
In the case when $T_a(b(a)-)=b(a)$ and $b(a)=b_k/a$, for some $k\ge1$, then since $T_{a'}(b(a))<b(a)$ for $a'>a$ sufficiently close to $a$, we 
derive that the support of the a.c.i.p. $T_{a'}$ is constantly equal to $K(a)$ for all $a'>a$ close to $a$. Thus, the only possible obstacle left is the set 
$\{a>1\ ;\ b(a)\ \text{is a fixed point for $T_a$}\}$. But using that the fixed points for $T_a$ are strictly decreasing in $a$ 
and $b(a)$ is non-decreasing in $a$,
we deduce that this set is countable and nowhere dense. This concludes the proof of Lemma~\ref{l.uniqueb}.

We proceed with the proof of Corollary~\ref{c.beta} which is an application of Lemma~\ref{l.startcalc}.
Disregarding countably many points we can cover $(1,\infty)$ by intervals $\I$ as described in the beginning of Section~\ref{s.beta}. 
Thus, in order to prove Corollary~\ref{c.beta}, it is sufficient to
verify the requirements of Lemma~\ref{l.startcalc} for the family $T_a$ together with the map $Y$ restricted to such a parameter interval $\I$. 
Recall that $\la>1$ stands for a uniform lower bound for the expansion in the family.
For $x\in[0,1]$, observe that by the definition of $T_a$, we have $\partial_aT_a(x)=T'(ax)x\ge0$, for all $a\in\I$ such that $x\neq b_k/a$, $k\ge0$.
Recall the formula \eref{eq.startcalc1} for the derivative of $y_j(a)=T_a^j(Y(a))$, $j\ge1$ (where we set $k=0$). 
Since $Y'(a)\ge0$ and $Y(a)>0$, for all $a\in\I$, since the points of discontinuity $b_k/a$ for the map $T_a$ are strictly decreasing, 
and since all the terms on the right-hand side of \eref{eq.startcalc1} are non-negative, we derive inductively that, 
for each $j\ge1$, the maps $y_j$ and $T_a^j$ are differentiable in $a$ and $Y(a)$, 
respectively, for all but finitely many $a\in\I$. Furthermore, since $\I$ is closed we have that $Y(a)$ is uniformly bounded away from $0$ and, 
from the term in the sum in \eref{eq.startcalc1} when $i=0$, we obtain
\begin{equation}
\label{eq.betastartcalc}
y_j'(a)\ge T_a^{j-1}\,'(Y(a))T'(Y(a))\inf_{a\in\I}Y(a)\ge\const\la^{j-1}.
\end{equation}
Thus, we find $j_0\ge0$ such that \eref{eq.startcomp} is satisfied. The only obstacle in applying Lemma~\ref{l.startcalc} 
might be that $y_{j_0}(a)\not\in K(a)$. 
However, by \eref{eq.betastartcalc} and property~\eref{eq.propertyK}, we derive that, 
by possibly disregarding a countable number of points, we can cover $\I$ by a 
countable number of intervals $\tilde{\I}\subset\I$ such that for each such interval $\tilde{\I}$ there is an integer $j\ge j_0$ such that 
$y_j|_{\tilde{\I}}$ is $C^1$ satisfying \eref{eq.startcomp} and $y_j(a)\in K(a)$ for all $a\in\tilde{\I}$. 
By Lemma~\ref{l.startcalc}, it follows that condition~(I) is satisfied for the map $X(a)=y_j(a)$, $a\in\tilde{\I}$. By 
Theorem~\ref{t.beta} this concludes the proof of Corollary~\ref{c.beta}.

\subsection{Condition (II)}
\label{ss.zweib}
The verification of condition (III) in the Section~\ref{ss.betacdrei} does not make use of condition~(II). Hence, by Lemma~\ref{l.app1}, 
we can without loss of generality assume that there is a constant $C=C(\I)\ge1$ such that 
for each $a\in\I$ the density $\varphi_a$ is bounded from above by $C$ and, further, there exists an interval 
$J(a)$ of length $C^{-1}$ such that $\varphi_a$ restricted 
to $J(a)$ is bounded from below by $C^{-1}$ (otherwise, disregarding a finite number of points, by 
Lemma~\ref{l.app1}, we can cover the interval $\I$ by a countable number of closed subintervals on each of which 
this is true and then proceed with these subintervals instead of $\I$). To conclude the verification of 
condition~(II) it is left to show that there exists a lower bound for $\varphi_a$ on the whole of $K(a)$.

To make the definition of the intervals $J_i(a)$ below work, we assume that the interval $J(a)$ is closed to the left and open to the right. 
Recall that, by property~(i) in Section~\ref{ss.prel}, we have $b_k/a-b_{k-1}/a\ge\delta_0$, $1\le k\le p_0$, for some
constant $\delta_0=\delta_0(\I)>0$. Let $\vep=\min\{(\la-1)/2C,\la\delta_0\}$ and
take $l\ge1$ so large that $\la^l/2C>1$. We claim that
$[0,\vep)\subset T_a^l(J(a))$. Let $J_0(a)=J(a)$ and assume that we
have defined the interval $J_{i-1}(a)\subset J(a)$, $i\ge1$, 
where $J_{i-1}(a)$ is a (not necessarily maximal) interval of monotonicity for $T_a^{i-1}$. 
If $[0,\vep)\subset T_a^i(J_{i-1}(a))$, we stop and do not define $J_i(a)$. If $[0,\vep)$ is not contained in 
$T_a^i(J_{i-1}(a))$ then, since $J_{i-1}(a)$ is a monotonicity interval for $T_a^{i-1}$ and by the 
definition of $\vep$, it follows that there can lie at most 
one partition point $b_k/a$ in the image $T_a^{i-1}(J_{i-1}(a))$. 
If there is no partition point in this image then we let 
$J_i(a)=J_{i-1}(a)$, which is in this case also a monotonicity interval for $T_a^i$. If there is a partition point 
$b_k/a\in T_a^{i-1}(J_{i-1}(a))$, then we define $J_i(a)\subset J_{i-1}(a)$ to be the 
interval of monotonicity for $T_a^i$ such that $T_a^{i-1}(J_i(a))=T_a^{i-1}(J_{i-1}(a))\cap[0,b_k/a)$. 
Note that $|T_a^{i-1}(J_{i-1}(a))\cap[b_k/a,1]|<\vep/\la$, since otherwise we would have 
$[0,\vep)\subset T_a^i(J_{i-1}(a))$. Assuming that $J_l(a)$ is defined, we obtain
\begin{align*}
|T_a^l(J_l(a))|&\ge\la(|T_a^{l-1}(J_{l-1}(a))|-\vep/\la)\ge\la^l|J_0(a)|-\vep\frac{\la^l-1}{\la-1}\\
&\ge\la^l(1/C-1/2C)\ge\la^l/2C>1,
\end{align*}
where we used the definitions of $\vep$ and $l$. Since $T_a^l(J_l(a))\subset[0,1]$, 
this is a contradiction and it follows that the maximal integer $i\ge0$ such that $J_i(a)$ is defined 
is strictly smaller than $l$. Hence, $T_a^l(J(a))$ contains $[0,\vep)$ as claimed above. 

Let $a'\in\I$ be the left boundary point of $\I$ and denote by $b$ the right boundary point of $K(a')$ (which is 
constant in $a\in\I$). As observed in the proof of Lemma~\ref{l.uniqueb} there exists a partition point 
$b_k/a'$, $k\ge1$, such that $b_k/a'<b$ and $T_{a'}(b_k/a'-)=b$.
Since $T_{a'}$ is exact (see \cite{wagner}) we derive that there exists an integer $l'<\infty$ 
such that $b_k/a'\in T_{a'}^{l'}([0,\vep))$. Since $T_{a'}^{l'}([0,\vep))\subset T_a^{l'}([0,\vep))$ for all $a>a'$, it 
follows that $T_a^{l'+1}([0,\vep))=[0,b)$ for all $a\in\I$.
Altogether, we derive that for $j\ge l+l'+1$ (i.e., $j$ is independent on $a\in\I$) we have $K(a)=\closure\{T_a^j(J(a))\}$
for all $a\in\I$.
Now, by the Perron-Frobenius equality, it follows that
\begin{equation}
\label{eq.perron22}
\varphi_a(y)\ge\sum_{\begin{subarray}{c}
              x\in J(a)\\
              T_a^j(x)=y
              \end{subarray}}\frac{\varphi_a(x)}{|T_a^j\,'(x)|}
\ge\frac{1}{C\Lambda^j},
\end{equation}
for a.e. $y\in K(a)$, $a\in\I$ (recall that $\Lambda$ is an upper bound for the maximal expansion in the family). 
This concludes the proof of a lower bound for $\varphi_a$ on the whole of $K(a)$.

\subsection{Condition (III)}
\label{ss.betacdrei}
We verify condition (III) by induction over $j\ge1$. Let $a_1,a_2\in\I$ such that 
$a_1\le a_2$. Note that $\P_1(a)$ consists of the elements $(b_{k-1}(a),b_k(a))$, $1\le k<p_0$ and the 
element $(b_{p_0-1}(a),b)$ where $b$ is the right boundary point of $K(a)$ (which does not depend on $a\in\I$). 
Thus, if $1\le k<p_0$ then we clearly have $T_{a_1}((b_{k-1}/a_1,b_k/a_1))=T_{a_2}((b_{k-1}/a_2,b_k/a_2))$. 
Since $b_{p_0-1}/a<b\le b_{p_0}/a$, for all $a\in\I$, we derive that
$T_{a_1}((b_{p_0-1}/a_1,b))\subset T_{a_2}((b_{p_0-1}/a_2,b))$.  Hence, (III) holds for $j=1$. Assume that (III) holds 
for $j\ge1$. Let $\ome\in\P_j(a_1)$ and $\ome'=\S_{a_1,a_2,j}(\ome)$ the corresponding element in $\P_j(a_2)$. 
Note that the image by $T_a^i$, $i\ge1$, of an element in $\P_i(a)$ is always adjacent to $0$.
Since $T_{a_1}^j(\ome)\subset T_{a_2}^j(\ome')$ and the partition points $b_k/a$'s are decreasing, 
it follows immediately that for every element $\om\in\P_{j+1}(a_1)|\ome$ there is 
a unique element $\om'\in\P_{j+1}(a_2)|\ome'$ fulfilling
$\ind_{a_1}(T_{a_1}^i(\om))=\ind_{a_2}(T_{a_2}^i(\om'))$, $0\le i<j+1$, and 
$T_{a_1}^{j+1}(\om)\subset T_{a_2}^{j+1}(\om')$.
Defining $\S_{a_1,a_2,j+1}(\om)=\om'$ shows that (III) holds also for $j+1$.

\section{Piecewise expanding unimodal maps}
\label{s.unimodal}
Let $T$ be a piecewise expanding unimodal map as defined below and $T_a$, $a\in\i$, $\delta>0$, a one-parameter family of 
piecewise expanding unimodal maps through $T$, i.e., $T_0=T$. In this section, we will show that if the family is 
non-degenerate (or transversal) then there exists $0<\vep\le\delta$ such that 
for Lebesgue almost every parameter values $a\in\ii$ the turning point of $T_a$ is typical for the a.c.i.p. for $T_a$. 
We will also state a condition for other points than the turning point in order to get almost sure typicality.

Henceforth, let $K>0$ be a fixed real number. We call a map $T:[-K,1]\to\real$ a {\em piecewise expanding unimodal\/} map 
if it satisfies the following properties: 
\begin{itemize}
\item $T\in C^0([-K,1])$ and $T$ is $C^{1,1}(L)$ on the intervals $[-K,0]$ and $[0,1]$.
\item $\inf_{x\in[-K,0]}T'(x)>1\quad\text{and}\quad\sup_{x\in[0,1]}T'(x)<-1$.
\item $T(0)=1$,\quad $T(1)>-K\quad\text{and}\quad T^2(1)\ge T(1)$.
\end{itemize}
Observe that the interval $[T(1),1]$ is forward invariant for $T$. 
Since $T$ is $C^{1,1}(L)$ on $[-K,1]$ except at the point $0$, 
by \cite{wong}, there exists a unique a.c.i.p. $\mu$ for $T$. 
We denote the set of piecewise expanding unimodal maps by $\U$. 
We say that a map $T\in\U$ is {\em mixing} if it is topologically mixing on $[T(1),1]$. 
Observe that if $T$ is mixing then the support of $\mu$ is the whole interval $[T(1),1]$. 
In this section we will consider only maps $T\in\U$ which are mixing. 
A {\em one-parameter family through $T\in\U$} is a family $T_a\in\U$, $a\in\i$, $\delta>0$, satisfying:
\begin{itemize}
\item $T_0=T$. 
\item For all $x\in[-K,1]$ the map $a\mapsto T_a(x)$ is $C^1(\i)$. 
\item The maps $a\mapsto T_a(x)$ and $a\mapsto\partial_xT_a(x)$ are $\Lip(L)$ on $\i$.
\end{itemize}

\begin{Rem}
Let $T_a\in\U$, $a\in\i$, be a one-parameter family of mixing piecewise expanding maps 
through a map $T\in\U$ and $\psi_a$ 
the affine map from $[T_a(1),1]$ onto $[0,1]$ with, say, positive derivative. 
Since $T_a(1)$ is Lipschitz in $a$ and since the length of the invariant interval $[T_a(1),1]$ is 
bounded from below by $\la>1$ (recall that $\la$ is a lower bound for the expansion in the family; cf. \eref{eq.la}) 
and from above by $1+K$, it is obvious that the family 
$\psi_a\circ T_a\circ\psi_a^{-1}:[0,1]\to[0,1]$, $a\in\i$, satisfies properties~(i)-(iii) in 
Section~\ref{ss.prel}. Henceforth, we will consider the family $T_a:[T_a(1),1]\to[T_a(1),1]$, $a\in\i$, and not the family 
defined on the unit interval which is affinely conjugated to it. The partitions in Section~\ref{ss.partitions} 
are defined in an analogous way for the family $T_a$.
\end{Rem}

To each one-parameter family $T_a\in\U$, $a\in\i$, through a map $T\in\U$ we associate a number $\Lambda_0\ge0$ 
given by
$$
\Lambda_0=\frac{\sup_{x\in[T(1),1]}\left|\partial_aT_a(x)|_{a=0}\right|}{\la-1}.
$$
This number serves as a threshold in order to get almost sure typicality. It stays in direct correspondence to the threshold given 
in Lemma~\ref{l.startcalc} in order to verify condition~(I). The main result of this section is the following.

\begin{Thm}
\label{t.main1}
Let $T_a\in\U$, $a\in\i$, $\delta>0$, be a one-parameter family of mixing unimodal maps
through a map $T\in\U$. If there exists $j_0\ge3$ such that $|D_aT_a^{j_0}(0)|_{a=0}|>\Lambda_0$, 
then there exists $0<\vep\le\delta$ such that $0$ is typical for $\mu_a$ for almost every $a\in\ii$. 
\end{Thm}

In order to prove Theorem~\ref{t.main1}, we will show that there exists $0<\vep\le\delta$ 
such that the family $T_a$, $a\in\ii$, together with the map $a\mapsto T_a^{j_0}(0)$, $a\in\ii$, satisfy 
conditions~(I)-(III). (To verify condition~(II) and (III) we possibly have to divide $\ii$ into smaller intervals.)
Knowing that conditions~(II) and (III) hold for a one-parameter family, we can apply Theorem~\ref{t.main}, 
and one directly gets almost sure typicality statements for other points than the turning point.

\begin{Cor}
\label{c.main2}
Let $T_a\in\U$, $a\in\i$, $\delta>0$, be a one-parameter family of mixing unimodal maps
through a map $T\in\U$ such that $|D_aT_a^{j_0}(0)|_{a=0}|>\Lambda_0$, for some $j_0\ge3$.
Then there exists $0<\vep\le\delta$ such that if $X:\ii\to[-K,1]$ is a $C^1$ map satisfying condition~(I) 
then $X(a)$ is typical for $\mu_a$ for almost every $a\in\ii$.
\end{Cor}

Before we start proving Theorem~\ref{t.main1}, 
we would like to point out the connection of the requirement $|D_aT_a^{j_0}(0)|_{a=0}|>\Lambda_0$ in 
Theorem~\ref{t.main1} and Corollary~\ref{c.main2} to a standard non-degeneracy condition for one-parameter families 
of maps on the interval. A one-parameter family through a map $T\in\U$ is {\em non-degenerate} or {\em transversal} if
$$
\sum_{i=1}^{\infty}\frac{\partial_aT_a(T^i(0))|_{a=0}}{T^i\,'(1)}\neq0,
$$
in the case where the turning point $0$ is not periodic for $T$, or if
$$
\sum_{i=1}^{p-1}\frac{\partial_aT_a(T^i(0))|_{a=0}}{T^i\,'(1)}\neq0,
$$
in the case where $0$ has prime period $p$ for $T$ (observe that $p\ge3$).
This non-degeneracy condition (or transversality condition) appears, e.g., in a generalization of Jakobson's Theorem in \cite{tsujii}. 
In the context of piecewise expanding unimodal maps it appears, e.g., in \cite{baladi}, \cite{bas1}, and \cite{bas2}.
We say that a piecewise expanding unimodal map $T$ is {\em good} if either $0$ is not periodic, or writing $p\ge3$ for the prime 
period of $0$, if
$$
|T^{p-1}\,'(1)|\min\{|T'(0-)|,|T'(0+)|\}>2.
$$
(See Remark~\ref{r.good} below.) 

\begin{Lem}
\label{l.nondegenerate}
Let $T_a\in\U$, $a\in\i$, $\delta>0$, be a one-parameter family of unimodal maps
through a good map $T\in\U$. Requiring that the family is non-degenerate is equivalent 
to requiring that there exists $j_0\ge3$ such that $|D_aT_a^{j_0}(0)|_{a=0}|>\Lambda_0$.
\end{Lem}

Hence, if $T_a\in\U$, $a\in\i$, $\delta>0$, is a non-degenerate one-parameter family of mixing unimodal maps
through a good map $T\in\U$ then, by Lemma~\ref{l.nondegenerate} and Theorem~\ref{t.main1}, 
we have almost sure typicality of the turning point for maps $T_a$ close to $T$. 
We prove Lemma~\ref{l.nondegenerate} in Section~\ref{ss.unieins}.



\subsection{Condition (I) and proof of Lemma~\ref{l.nondegenerate}}
\label{ss.unieins}
We verify first that there exists $0<\vep\le\delta$ such that 
condition~(I) holds for the map $a\mapsto T_a^{j_0}(0)$, $a\in\ii$.
We apply the criteria in Lemma~\ref{l.startcalc}. 
We consider first the case when $T^i(0)\neq0$ for all $0<i<j_0$. It follows that $T_a^i(0)\neq0$ for all $0<i<j_0$ and 
all $a$ sufficiently close to $0$. Hence, $a\mapsto T_a^{j_0}(0)$ is differentiable close to $0$ where 
$|D_aT_a^{j_0}(0)|$ is close to $|D_aT_a^{j_0}(0)|_{a=0}|>\Lambda_0$.
This implies that there exist $0<\vep\le\delta$ and $\kappa>0$ such that, for all $a\in\ii$,
\begin{equation}
\label{eq.signuni}
|D_aT_a^{j_0}(0)|\ge\frac{\sup_{a\in\ii}\sup_{x\in[T_a(1),1]}|\partial_aT_a(x)|}{\la-1}+\kappa.
\end{equation}
Observe that the constant $2L$ in \eref{eq.startcomp} is to ensure 
that the partition points $b_k(a)$, $0\le k\le p_0$, do not move faster than the points $y_j(a)$, $j\ge j_0$, 
in Lemma~\ref{l.startcalc}. But in the setting here the turning point $0$ does not change in $a$ which 
implies that it is sufficient to have a constant $\kappa>0$ instead of $2L$.
Hence, by Lemma~\ref{l.startcalc}, it follows that the one-parameter family $T_a$, $a\in\ii$, 
together with the map $a\mapsto T_a^{j_0}(0)$ satisfies condition~(I). 
Regarding the verification of condition~(III) in Section~\ref{ss.cdreiuni} we observe that 
the proof of Lemma~\ref{l.startcalc} in fact implies that for either $\sigma\equiv+1$ or 
$\sigma\equiv-1$ ($\sigma$ is independent on the choice of $j$ and $a$),
\begin{equation}
\label{eq.signconst}
\sign\left(\frac{D_aT_a^j(0)}{T_a^{j-j_0}\,'(T_a^{j_0}(0))}\right)\equiv\sigma,
\end{equation}
for all $j\ge j_0$ and all $a\in\ii$ for which the derivative $D_aT_a^j(0)$ exists.
Concerning the verification of condition~(I), 
it is left to consider the case when $0$ is periodic for $T$ with prime period $p\ge3$ and $p<j_0$. 
We consider first the a-priori possible case when there is $\tilde{p}<j_0$ such that $T^{\tilde{p}}(0)=0$ 
and $|D_aT_a^{\tilde{p}}(0)|_{a=0}|=0$. Let $\tilde{p}$ be minimal with this property.
Since $0$ is periodic for $T$ with (not necessarily prime) period $\tilde{p}\ge3$, we derive from 
\eref{eq.startcalc1} (set $y_j(a)=T_a^j(0)$) the formula
\begin{equation}
\label{eq.deg1}
|D_aT_a^{l\tilde{p}}(0)|_{a=0}|=|D_aT_a^{\tilde{p}}(0)|_{a=0}|
\underbrace{\lim_{a\to0+}\left|1+\sum_{i=1}^{l-1}\prod_{k=1}^iT_a^{\tilde{p}}\,'(T_a^{k\tilde{p}}(0))\right|}_{(*)},
\end{equation}
for $l>1$. If $l_0\ge1$ is maximal such that $l_0\tilde{p}<j_0$, by equation~\eref{eq.deg1}, we obtain that also $|D_aT_a^{l_0\tilde{p}}(0)|_{a=0}|=0$.
Thus, applying \eref{eq.startcalc1} with $k=l_0\tilde{p}$ we derive that 
$|D_aT_a^{j_0-l_0\tilde{p}}(0)|_{a=0}|=|D_aT_a^{j_0}(0)|_{a=0}|>\Lambda_0$. 
This in turn implies that we can assume that $|D_aT_a^{lp}(0)|_{a=0}|\neq0$ for all $l\ge1$ such that $lp<j_0$,
which in turn implies that $T_a^i(0)\neq0$ for all $0<i<j_0$ and all $a\neq0$ sufficiently close to $0$. Hence, we are 
in a similar setting as in the first case above 
and we get that condition~(I) is satisfied for the map $a\mapsto T_a^{j_0}(0)$ when $a$ is sufficiently close to $0$. 
This concludes the verification of condition~(I).

We turn to the proof of Lemma~\ref{l.nondegenerate}. 
If $0$ is non-periodic for $T$ observe that
\begin{equation}
\label{eq.degstart}
\sum_{j=0}^{\infty}\frac{\partial_aT_a(T^j(0))|_{a=0}}{T^{j}\,'(1)}
=\lim_{j\to\infty}\frac{D_aT_a^{j+1}(0)|_{a=0}}{T^j\,'(1)}.
\end{equation}
If the family is non-degenerate then $|D_aT_a^{j+1}(0)|_{a=0}|$ is growing in $j$ as $|T^j\,'(1)|$ and we find 
$j_0\ge3$ such that $|D_aT_a^{j_0}(0)|_{a=0}|>\Lambda_0$. In the other direction if 
$|D_aT_a^{j_0}(0)|_{a=0}|>\Lambda_0$ for some $j_0\ge3$, we have seen in the first part of this section 
that this implies that condition~(I) is satisfied for $a$ close to $0$. By \eref{eq.degstart}, this shows that 
the family is non-degenerate. 

If $0$ has prime period $p\ge3$ then since
$$
0\neq\sum_{i=1}^{p-1}\frac{\partial_aT_a(T^i(0))|_{a=0}}{T^i\,'(1)}=\frac{D_aT_a^p(0)|_{a=0}}{T^{p-1}\,'(1)},
$$
it follows that $D_aT_a^p(0)|_{a=0}\neq0$. 
Since the map $T_0$ is good it follows that $(*)$ in equality \eref{eq.deg1} (set $\tilde{p}=p$) is growing exponentially in $l$, and 
we can find $j_0\ge3$ such that $|D_aT_a^{j_0}(0)|_{a=0}|>\Lambda_0$. On the other hand 
if $|D_aT_a^{j_0}(0)|_{a=0}|>\Lambda_0$ for some $j_0\ge0$ then we have seen above that this implies 
$D_aT_a^p(0)|_{a=0}\neq0$. This concludes the proof of Lemma~\ref{l.nondegenerate}.

\begin{Rem}
\label{r.good}
The fact that $T$ is good is only used to guarantee that $(*)$ (when $\tilde{p}=p$) is growing to $+\infty$ when $l$ increases. We 
could replace the condition that $T$ is good by the requirement that $(*)\to\infty$ as $l\to\infty$. 
\end{Rem}

\subsection{Condition (II)}
In Section~\ref{ss.cdreiuni} it is shown that we can cover $\ii$ by a countable number of closed intervals $\I$ such that 
condition~(III) is satisfied for the family $T_a$, $a\in\I$. 
Since Section~\ref{ss.cdreiuni} does not make use of condition~(II), by Lemma~\ref{l.app1}, 
we can without loss of generality assume that there is a constant $C=C(\I)\ge1$ such that 
for each $a\in\I$ the density $\varphi_a$ is bounded from above by $C$ and, further, there exists an interval 
$J(a)$ of length $C^{-1}$ such that $\varphi_a$ restricted 
to $J(a)$ is bounded from below by $C^{-1}$ (otherwise, disregarding a finite number of points, by 
Lemma~\ref{l.app1}, we can cover the interval $\I$ by a countable number of closed subintervals on each of which 
this is true and then proceed with these subintervals instead of $\I$). 
It follows that, there exists an integer $0\le i\le\ln(3C(1+K))/\ln\la$ 
such that $0\in T_a^i(J(a))$. 
By \cite{wagner} and since $T_a$ is mixing, we have that $T_a:[T_a(1),1]\to[T_a(1),1]$ is exact, i.e., for each set $S\subset [T_a(1),1]$ 
of positive Lebesgue measure it follows that $\lim_{j\to\infty}|[T_a(1),1]\setminus T_a^j(S)|=0$.
Observe that if $T\in\U$ and $I\subset[T(1),1]$ is an interval of length close to $1+|T(1)|$ then we have $T^2(I)=[T(1),1]$. 
Thus, exactness implies that there is an integer $k$ such that $T_a^k([-1/2C,0])=T_a^k([0,1/2C])=[T_a(1),1]$. 
Since the image of an interval by $T_a^j$, $j\ge1$, changes continuously in $a$ and since the parameter interval $\I$ is closed 
we can in fact choose the integer $k$ above independently on $a\in\I$. Hence, for $j\ge k+\ln(3C(1+K))/\ln\la$, 
we have $T_a^j(J(a))=[T_a(1),1]$, for all $a\in\I$. 
By \eref{eq.perron22}, we get a uniform lower bound for the density. This concludes the verification of condition~(II).

\subsection{Condition (III)}
\label{ss.cdreiuni}
As shown in Section~\ref{ss.unieins} there exists $0<\vep'\le\delta$ 
such that condition~(I) is satisfied for the family $T_a$, $a\in[0,\vep']$, together with 
the map $a\mapsto X(a)=T_a^{j_0}(0)$. 
If $0$ is non-periodic for $T$, we can choose $0<\vep\le\vep'$ so small that 
\begin{equation}
\label{eq.j0}
\dist(0,T_a^j(0))=:\kappa>0,\quad\text{for all }1\le j\le j_0,
\end{equation}
for all $a\in\ii$.
If $0$ is periodic for $T$ then we have shown in Section~\ref{ss.unieins} that if $1\le j\le j_0$ such that $T^j(0)=0$ 
then $|D_aT_a^j(0)|_{a=0}|\neq0$. Hence, there is $0<\vep\le\vep'$ such that, disregarding the point $0$, 
we can cover the interval $\ii$ by a countable number of closed intervals $\I$ such that for 
each such interval $\I$ there exists a constant $\ka=\ka(\I)>0$ such that \eref{eq.j0} holds for all $a\in\I$. 
Henceforth, in the periodic case fix such an interval $\I$, and in the non-periodic case we use $\I$ to denote the interval 
$\ii$. For $a\in\I$, observe that if $\om\in\P_j(a)$, $j\ge1$ (note that $\P_j(a)$ is the partition on the phase space), 
then the image of $\om$ by $T_a^j$ is of the form
\begin{equation}
\nonumber
T_a^j(\om)=(T_a^k(0),T_a^l(0)),\quad\text{for some}\ 1\le k,l\le j+2.
\end{equation}
The choice of the integers $k$ and $l$ might not be unique due to the possibility that $0$ is periodic or pre-periodic for $T_a$. 
If $b\in\partial\om$ and $T_a^j(b)=T_a^k(0)$, $1\le k\le j+2$, we say that $k$ is {\em minimal\/} if the following holds. 
If $k\le j$ then $k$ is minimal if $T_a^{j-k}(b)=0$ but $T_a^{j-i}(b)\neq0$ for all $1\le i<k$. 
If $k=j+1$ or $k=j+2$ then $k$ is minimal if $b=T_a^{k-j}(0)$ and $T_a^i(b)\neq0$, for all $0\le i<k$. 
Obviously for each element $\om\in\P_j(a)$ there exist unique minimal integers $1\le k,l\le j+2$ such that 
$T_a^j(\om)=(T_a^k(0),T_a^l(0))$.

Let $\P_j|\I$, $j\ge1$, be the partition in the parameter space associated to the map $X$. 
Recall that $x_j(a)=T_a^j(X(a))=T_a^{j+j_0}(0)$. By \eref{eq.signconst}, it follows that 
for all $\om\in\P_j|\I$ and $a\in\om$ we either have $\sign(x_i'(a))=\sign(T_a^{j_0+i-1}\,'(1))$ 
or $\sign(x_i'(a))=-\sign(T_a^{j_0+i-1}\,'(1))$, for all $0\le i\le j$.
Without loss of generality we assume that we are in the first case, i.e., 
\begin{equation}
\label{eq.signchange}
\sign(x_i'(a))=\sign(T_a^{j_0+i-1}\,'(1)).
\end{equation}
(In the second case we would have to require $a_2\le a_1$ instead of $a_1\le a_2$ in the statement of condition~(III).)
For $a\in\I$, we claim that if $\om\in\P_j(a)$, $j\ge1$, and $k,l\ge1$ are chosen minimal 
such that $T_a^j(\om)=(T_a^k(0),T_a^l(0))$ then $T_a^k(0)$ and $T_a^l(0)$ are differentiable in $a$ and
\begin{itemize}
\item[a)]
if $k\ge j_0$ then $D_a T_a^k(0)<0$; and
\item[b)] if $l\ge j_0$ then $D_a T_a^l(0)>0$.
\end{itemize}
We consider the case a). Since $k$ is chosen minimal we have that $T_a^i(0)\neq0$ for all $1\le i<k$ which ensures that 
the point $T_a^k(0)$ is differentiable in $a$. Since $T_a^k(0)$ is the left boundary point of the image $T_a^j(\om)$, 
it is easy to verify that we must have $T_a^{k-1}\,'(1)<0$ (for example use induction over $j$). By \eref{eq.signchange}, it follows 
that $D_a T_a^k(0)=x_{k-j_0}'(a)<0$ as claimed. The case b) is shown similarly. In order to verify condition~(III) we 
will first prove the following lemma.

\begin{Lem}
\label{l.helptwo}
Let $a_1,a_2\in\I$ and $a_1\le a_2$. For all $j\ge1$ there is a mapping
$$
\S_{a_1,a_2,j}:\P_j(a_1)\to\P_j(a_2),
$$
such that, for all $\om\in\P_j(a_1)$, 
\begin{equation}
\label{eq.drei11}
\ind_{a_1}(T_{a_1}^i(\om))=\ind_{a_2}(T_{a_2}^i(\S_{a_1,a_2,j}(\om))),\quad0\le i<j. 
\end{equation}
Furthermore, for all $\om\in\P_j(a_1)$ we have the following. Let $k_1,l_1,k_2,l_2\ge 1$ be the minimal 
integers such that 
$$
T_{a_1}^j(\om)=(T_{a_1}^{k_1}(0),T_{a_1}^{l_1}(0))\quad\text{and}\quad T_{a_2}^j(\S_{a_1,a_2,j}(\om))
=(T_{a_2}^{k_2}(0),T_{a_2}^{l_2}(0)).
$$
Then, we have $k_1\ge k_2$ and $T_{a_1}^{k_2}(0)\le T_{a_1}^{k_1}(0)$; and, similarly, we have 
$l_1\ge l_2$ and $T_{a_1}^{l_1}(0)\le T_{a_1}^{l_2}(0)$.
\end{Lem}

\begin{proof}
We proof Lemma \ref{l.helptwo} by induction over $j$. Observe that, by \eref{eq.j0}, it easily follows that for 
$1\le j<j_0$ and for all $a_1,a_2\in\I$, there is a bijection 
$$
\S_{a_1,a_2,j}:\P_j(a_1)\to\P_j(a_2),
$$
such that Lemma~\ref{l.helptwo} is satisfied (in this case the assumption $a_1\le a_2$ is not necessary and 
one has always $k_1=k_2$ and $l_1=l_2$). 
Thus, let $j\ge j_0-1$ and assume that the assertion of 
Lemma~\ref{l.helptwo} holds for $j$ and all $a_1,a_2\in\I$ such that $a_1\le a_2$. We will show that the assertion also holds for 
$j+1$. 

Let $a_1,a_2\in\I$, $a_1\le a_2$, and for $\om_1\in\P_j(a_1)$ and $\om_2=\S_{a_1,a_2,j-1}(\om_1)$ let
$k_1,k_2,l_1,l_2\ge1$ be the corresponding minimally chosen integers. Observe that by 
the induction assumption we have $k_1\ge k_2$ and $l_1\ge l_2$.

\begin{Claim}
We have
\begin{equation}
\label{eq.uuuni}
T_a^i(0)\neq0,\quad\text{for all $a\in[a_1,a_2]$ and for all $1\le i<\max\{k_2,l_2\}$}.
\end{equation}
In particular this implies that the maps $a\mapsto T_a^{k_2}(0)$ and $a\mapsto T_a^{l_2}(0)$ are differentiable on the interval $[a_1,a_2]$. 
Furthermore, we claim that if $k_2\ge j_0$ then $D_a T_a^{k_2}(0)<0$, and if $l_2\ge j_0$ then $D_a T_a^{l_2}(0)>0$.
\end{Claim}

\begin{proof}
For $a\in[a_1,a_2]$, set $\om=\S_{a_1,a,j}(\om_1)$, and let $k,l\ge1$ be the 
associated minimal integers for $\om$. 
Observe that since $\S_{a,a_2,j}(\om)=\om_2$ it follows from the induction assumption that $k\ge k_2$ and $l\ge l_2$. 
By the minimality of $k$ and $l$ it follows that $T_a^i(0)\neq0$ for all $1\le i<\max\{k_2,l_2\}$. 
Thus, this implies \eref{eq.uuuni}.

If $k_2\ge j_0$ then \eref{eq.uuuni} implies that $[a_1,a_2]$ is contained in an element of $\P_{k_2-j_0}|\I$. 
By property a) and \eref{eq.signchange} we have that $x_{k_2-j_0}'(a_2)=D_a T_a^{k_2}(0)|_{a=a_2}<0$. Since 
condition~(I) is satisfied we obtain that $D_a T_a^{k_2}(0)<0$ for all $a\in[a_1,a_2]$. By a similar argument using 
property b) we obtain that $D_a T_a^{l_2}(0)>0$ for all $a\in[a_1,a_2]$, if $l_2\ge j_0$.
\end{proof}

We claim that 
\begin{align}
\label{eq.b111}
\text{if $T_{a_1}^{k_1}(0)<T_{a_2}^{k_2}(0)$},\quad &\text{then $0\not\in[T_{a_1}^{k_1}(0),T_{a_2}^{k_2}(0)]$, and} \\
\label{eq.b112}
\text{if $T_{a_2}^{l_2}(0)<T_{a_1}^{l_1}(0)$},\quad &\text{then $0\not\in[T_{a_2}^{l_2}(0),T_{a_1}^{l_1}(0)]$}.
\end{align}
Regarding \eref{eq.b111}, observe that if $k_2\ge j_0$ then 
the claim above implies that $T_{a_1}^{k_2}(0)\ge T_{a_2}^{k_2}(0)$. 
But by the induction assumption we have 
$T_{a_1}^{k_2}(0)\le T_{a_1}^{k_1}(0)<T_{a_2}^{k_2}(0)$, and thus the case $k_2\ge j_0$ is not possible. 
For the case $k_2<j_0$ observe that, by \eref{eq.j0}, $T_a^i(0)\neq0$ for all $0\le i\le j_0$ and $a\in\I$. 
Since the image of $\I$ by the map $a\mapsto T_a^{k_2}(0)$ contains $[T_{a_1}^{k_2}(0),T_{a_2}^{k_2}(0)]$ and since 
$T_{a_1}^{k_2}(0)\le T_{a_1}^{k_1}(0)$, this shows \eref{eq.b111}. 
In a similar way one verifies \eref{eq.b112}.
It follows that if $0\not\in T_{a_1}^j(\om_1)\cup T_{a_2}^j(\om_2)$ then 
$T_{a_1}^j(\om_1)$ and $T_{a_2}^j(\om_2)$ lie both either to the left or to the right of $0$. 
This implies that $\om_1$ and $\om_2$ are also 
elements of $\P_{j+1}(a_1)$ and $\P_{j+1}(a_2)$, respectively. We consider the case when 
$T_{a_1}^j(\om_1)$ and $T_{a_2}^j(\om_2)$ lie to the left of $0$ (the other case is treated similarly). 
Let $1\le \tilde{k}_1,\tilde{k}_2,\tilde{l}_1,\tilde{l}_2\le j+3$ be the corresponding minimal integers for $\om_1$ and 
$\om_2$ (seen as elements in the $j+1$-th partitions).
We immediately get that $\tilde{k}_1=k_1+1$ and $\tilde{k}_2=k_2+1$.
Since by the induction assumption $T_{a_1}^{k_2}(0)\le T_{a_1}^{k_1}(0)$ and 
since $T_{a_1}^j(\om_1)$ is to the left of $0$ where $T_{a_1}$ has positive slope, we have also 
$T_{a_1}^{\tilde{k}_2}(0)\le T_{a_1}^{\tilde{k}_1}(0)$.
The situation for the right boundary points is slightly more 
difficult. First observe that if $l_1>l_2$ then 
\begin{equation}
\label{eq.uni111}
T_{a_1}^{l_2}(0)<0.
\end{equation}
If $l_2\le j_0$ this follows immediately from \eref{eq.j0}. If $l_2>j_0$ 
then we obtain from the claim above that $T_{a_1}^{l_2}(0)<T_{a_2}^{l_2}(0)$. 
Since $T_{a_2}^{l_2}(\om_2)$ lies to the left of $0$, this implies \eref{eq.uni111}.
By the induction assumption we get $T_{a_1}^{l_1}(0)\le T_{a_1}^{l_2}(0)<0$ and it follows that $\tilde{l}_1=l_1+1$. 
The integer $\tilde{l}_2$ is equal to $1$ if $T_{a_2}^{l_2}(0)=0$, and $\tilde{l}_2=l_2+1$ otherwise. In both cases, since $T_{a_1}$ 
has positive slope on the left of $0$, we obtain $T_{a_1}^{\tilde{l}_1}(0)\le T_{a_1}^{\tilde{l}_2}(0)$. Furthermore, $\tilde{l}_1\ge\tilde{l}_2$ holds. 
Hence, setting $\S_{a_1,a_2,j+1}(\om_1)=\om_2$, 
we have shown that the assertion of Lemma~\ref{l.helptwo} is satisfied in the case when $0\not\in T_{a_1}^j(\om_1)\cup T_{a_2}^j(\om_2)$. 

If $0\in T_{a_1}^j(\om_1)\cap T_{a_2}^j(\om_2)$, 
then each of the partitions $\P_{j+1}(a_1)|\om_1$ and $\P_{j+1}(a_2)|\om_2$ 
contain two elements $\om_{11},\om_{12}$ and $\om_{21},\om_{22}$, respectively. 
Let $\om_{11}$ and $\om_{21}$ be the elements which are mapped 
after $j$ iterations, say, to the left of $0$. Looking at the to $\om_{11}$ and $\om_{21}$ 
corresponding minimal integers $\tilde{k}_1,\tilde{l}_1,\tilde{k}_2,\tilde{l}_2\ge 1$, 
we obviously have $\tilde{l}_1=\tilde{l}_2=1$, and $\tilde{k}_1=k_1+1$ and $\tilde{k}_2=k_2+1$. 
Since by the induction assumption $T_{a_1}^{k_2}(0)\le T_{a_1}^{k_1}(0)$ and 
since $T_{a_1}^j(\om_1)$ is to the left of $0$ where $T_{a_1}$ has positive slope, we have also 
$T_{a_1}^{k_2+1}(0)\le T_{a_1}^{k_1+1}(0)$.
An analogue situation appears for $\om_{12}$ and $\om_{22}$. 
Thus, setting $\S_{a_1,a_2,j+1}(\om_{11})=\om_{21}$ and $\S_{a_1,a_2,j+1}(\om_{12})=\om_{22}$, 
shows that the assertion of Lemma~\ref{l.helptwo} in the case when $0\in T_{a_1}^j(\om_1)\cap T_{a_2}^j(\om_2)$.

The only case left is when $0\in T_{a_2}^j(\om_2)$ but $0\not\in  T_{a_1}^j(\om_1)$. 
(Observe that the case $0\not\in T_{a_2}^j(\om_2)$ but $0\in  T_{a_1}^j(\om_1)$ is excluded by \eref{eq.b111} and \eref{eq.b112}.)
We have that $\om_1$ is also an element of 
$\P_{j+1}(a_1)$. Without loss of generality assume that $T_{a_1}^j(\om_1)$ lies to to the left of $0$. If $\om_{21}$ 
is the element in $\P_{j+1}(a_2)|\om_2$ such that $T_{a_2}^j(\om_{21})$ lies to the left of $0$, 
then setting $\S_{a_1,a_2,j}(\om_1)=\om_{21}$ we derive with a similar reasoning as in 
the case when $0\not\in T_{a_1}^j(\om_1)\cup T_{a_2}^j(\om_2)$ 
that the assertion of Lemma~\ref{l.helptwo} is satisfied also in this last setting. 
\end{proof}

To verify that condition (III) is satisfied it is only left to show that there exists a constant $C_2$ such that 
properties \eref{eq.dreidist} and \eref{eq.drei2} are fulfilled. Let $a_1,a_2\in\I$, $a_1\le a_2$, take $\om_1\in\P_j(a_1)$, $j\ge1$, 
and set $\om_2=\S_{a_1,a_2,j}(\om_1)$ where $\S_{a_1,a_2,j}$ is given by Lemma~\ref{l.helptwo}. Let $1\le k_1,k_2,l_1,l_2\le j+2$ 
be the minimal integers corresponding to $\om_1$ and $\om_2$, respectively. Regarding \eref{eq.dreidist}, the only case to consider is 
when $T_{a_1}^j(\om_1)\cap T_{a_2}^j(\om_2)=\emptyset$. We consider the case when $T_{a_1}^j(\om_1)$ lies to the left of 
$T_{a_2}^j(\om_2)$ (the other case is similar). By Lemma~\ref{l.helptwo} it follows that $T_{a_1}^{k_2}(0)\le T_{a_1}^{k_1}(0)<T_{a_2}^{k_2}(0)$. 
By the claim in the proof of Lemma~\ref{l.helptwo} the situation that $T_{a_1}^{k_2}(0)<T_{a_2}^{k_2}(0)$ is only possible 
when $k_2<j_0$. It follows that
$$
|T_{a_1}^{k_1}(0)-T_{a_2}^{k_2}(0)|\le|T_{a_1}^{k_2}(0)-T_{a_2}^{k_2}(0)|\le\max_{1\le i<j_0}\max_{a\in\I}|D_a T_a^i(0)|\cdot|a_2-a_1|,
$$
which implies \eref{eq.dreidist}.

Regarding \eref{eq.drei2}, observe that if $\min\{k_2,l_2\}\ge j_0$, then, since $T_{a_1}^{k_2}(0)\le T_{a_1}^{k_1}(0)$ 
and $T_{a_1}^{l_1}(0)\le T_{a_1}^{l_2}(0)$, it follows immediately from the 
claim in the proof of Lemma~\ref{l.helptwo} that $T_{a_1}^j(\om_1)$ is entirely contained in $T_{a_2}^j(\om_2)$ 
which implies \eref{eq.drei2} with the constant $C_2$ equal to $1$. 
Furthermore, if $j<j_0$ then for all $\om\in\P_j(a)$, $a\in\I$, we have, by \eref{eq.j0}, that 
$|T_a^j(\om)|\ge\kappa$ which implies \eref{eq.drei2} with a constant $C_2$ equal to $(1+K)\kappa^{-1}$. 
So the only case left is when $\min\{k_2,l_2\}<j_0$ and $j\ge j_0$.  
Assume that $k_2=\min\{k_2,l_2\}$ (the case when $l_2=\min\{k_2,l_2\}$ is treated similarly).
Consider the images $T_{a_1}^{j-k_2+1}(\om_1)$ and $T_{a_2}^{j-k_2+1}(\om_2)$. 
The boundary points of $T_{a_1}^{j-k_2+1}(\om_1)$ are $T_{a_1}^{k_1-k_2+1}(0)$ and $T_{a_1}^{l_1-k_2+1}(0)$, and 
the boundary points of $T_{a_2}^{j-k_2+1}(\om_1)$ are $1$ and $T_{a_2}^{l_2-k_2+1}(0)$. 
Obviously, $T_{a_1}^{l_1-k_2+1}(0)$ and $T_{a_2}^{l_2-k_2+1}(0)$ are the left boundary points. 
If $T_{a_1}^{l_1-k_2+1}(0)\ge T_{a_2}^{l_2-k_2+1}(0)$ then $T_{a_1}^{j-k_2+1}(\om_1)$ is contained in $T_{a_2}^{j-k_2+1}(\om_2)$ 
which, by the distortion estimate in Lemma~\ref{l.distortion}, implies \eref{eq.drei2} 
with a constant $C_2$ equal to $C_3^2$. 
(Note that Lemma~\ref{l.distortion} requires that condition~(III) is satisfied. But in its proof it is only needed that 
property~\eref{eq.dreidist} holds which we have already verified above.)
The case left is when $T_{a_1}^{l_1-k_2+1}(0)<T_{a_2}^{l_2-k_2+1}(0)$. 
Since $k_2\le l_2$ it follows that $T_{a_1}^{j-i}(\om_1)$ and $T_{a_2}^{j-i}(\om_2)$ are not adjacent to $0$ for all $1\le i<k_2$. 
Hence, $\om_1$ and $\om_2$ are in fact also elements of $\P_{j-k_2+1}(a_1)$ and $\P_{j-k_2+1}(a_2)$, respectively, 
and the minimal integers $1\le\tilde{k}_1,\tilde{k}_2\le j-k_2+3$ such that $T_{a_1}^{j-k_2+1}(\om_1)=(T_{a_1}^{\tilde{k}_1}(0),*)$ 
and $T_{a_2}^{j-k_2+1}(\om_2)=(T_{a_2}^{\tilde{k}_2}(0),*)$ are given by $\tilde{k}_1=l_1-k_2$ and $\tilde{k}_2=l_2-k_2$.
Thus, we can apply Lemma~\ref{l.helptwo} and we obtain that 
$T_{a_1}^{l_2-k_2+1}(0)\le T_{a_1}^{l_1-k_2+1}(0)<T_{a_2}^{l_2-k_2+1}(0)$. By the claim in 
the proof of Lemma~\ref{l.helptwo} this implies that $l_2-k_2+1<j_0$. By \eref{eq.j0}, it follows that 
$1-T_{a_2}^{l_2-k_2+1}(0)\ge\kappa$ and we obtain 
that $|T_{a_2}^{j-k_2+1}(\om_2)|\ge\kappa$. Hence, \eref{eq.drei2} is satisfied if we choose the constant $C_2$ equally to 
$(1+K)\kappa^{-1}$. This concludes the verification of condition~(III) for the family $T_a$, $a\in\I$.

\section{Tent maps}
\label{s.tent}
In this section we apply the results of Section~\ref{s.unimodal} to an important family of piecewise expanding unimodal maps, 
the so-called skew tent maps. We consider particular one-parameter families of skew tent maps for which it is easy to make general 
statements. However, given a concrete one-parameter family of skew tent maps which does not fit into the families 
considered below, by Theorem~\ref{t.main1}, it is straightforward to check if one has almost sure typicality.

We will use the same representation 
as in \cite{mv}, i.e., we define the skew tent map with slopes $\al$ and $-\be$ where $\al,\be>1$, by the 
formula
$$
T_{\al,\be}(x)=\left\{\begin{array}{ll}
              1+\al x\quad&\text{if}\ x\le0,\\
              1-\be x\quad&\text{otherwise}.
              \end{array}\right.
$$
The turning point of $T_{\al,\be}$ is $0$, $T_{\al,\be}(0)=1$ and, by Lemma~3.1 in \cite{mv}, if 
$\al^{-1}+\be^{-1}\ge1$ then the interval $[T_{\al,\be}(1),1](=[1-\be,1])$ is invariant under $T_{\al,\be}$ 
(if $\al^{-1}+\be^{-1}<1$ then there exists no invariant interval of finite positive length). For two parameter couples $(\al,\be)$ and 
$(\al',\be')$ we take the same order relation as the one which appears in \cite{mv}, i.e.,
we shall write $(\al',\be')>(\al,\be)$ if $\al'\ge\al$, $\be'\ge\be$, and at least one of these inequalities is sharp.
Fix $(\al_0,\be_0)$ and $(\al_1,\be_1)$ in the set $\{(\al,\be)\ ;\ \al,\be>1\ \text{and}\ \al^{-1}+\be^{-1}\ge1\}$ 
such that $(\al_1,\be_1)>(\al_0,\be_0)$. Let
$$
\al:[0,1]\to[\al_0,\al_1]\quad\text{and}\quad\be:[0,1]\to[\be_0,\be_1]
$$
be functions in $C^1([0,1])$ such that $(\al(0),\be(0))=(\al_0,\be_0)$, $(\al(1),\be(1))=(\al_1,\be_1)$, 
and, for all $a\in[0,1]$, if $\al_0\neq\al_1$ then $\al'(a)>0$ and if $\be_0\neq\be_1$ then $\be'(a)>0$. 
Observe that $\al(a),\be(a)>1$, and $\al(a)^{-1}+\be(a)^{-1}\ge1$, for all $a\in[0,1]$. We define the 
one-parameter family $T_a$ as the family of skew tent maps given by 
$$
T_{\al(a),\be(a)}:[T_{\al(a),\be(a)}(1),1]\to[T_{\al(a),\be(a)}(1),1],\quad\ a\in[0,1].
$$  
Observe that the family $T_a$, $a\in[0,1]$, is a one-parameter family of piecewise expanding unimodal maps 
as described in Section~\ref{s.unimodal} where we can set the 
constant $K$ equal to $\max_{a\in[0,1]}\be(a)-1$. The main statement of this section is the following.

\begin{Thm}
\label{t.tent}
For a.e. parameter $a\in[0,1]$ the turning point $0$ is typical for the a.c.i.p. $\mu_a$.
\end{Thm}

For other points than the turning point, i.e., given a $C^1$ function $Y:[0,1]\to\real$ (such that $Y(a)\in[T_a(1),1]$), 
it is sufficient to check condition~(I) in order to obtain a.s. typicality for $Y$.

\begin{Cor}
\label{c.tent1}
If the one-parameter family $T_a$, $a\in[0,1]$, with the associated map $a\mapsto Y(a)$ 
satisfies condition (I), then $Y(a)$ is typical for $\mu_a$, for a.e. $a\in[0,1]$.
\end{Cor}


The following comments show that we can restrict ourself to non-renormalizable skew tent maps. If 
\begin{equation}
\label{eq.unirenorm}
\al\le\be/(\be^2-1),
\end{equation} 
then $T_{\al,\be}$ is renormalizable, see, e.g., \cite{mv}. More precisely if \eref{eq.unirenorm} holds, then 
$T_{\al,\be}^2(1)$ is greater or equal than the unique fixed point in $(0,1)$ and $T_{\al,\be}^2$ restricted either 
to the interval $[T_{\al,\be}(1),T_{\al,\be}^3(1)]$ or to the interval $[T_{\al,\be}^2(1),1]$ is affinely conjugated 
to $T_{\be^2,\al\be}$ restricted to the interval $[T_{\be^2,\al\be}(1),1]$. Observe that the new slopes 
$\al'=\be^2$ and $-\be'=-\al\be$ 
still satisfy $\al',\be'>1$ and $(\al')^{-1}+(\be')^{-1}\ge1$ (the latter inequality follows by \eref{eq.unirenorm}). 
Since the function $\be\mapsto\be/(\be^2-1)$ is decreasing for $\be>1$, 
we have that if $T_0$ is not renormalizable then $T_a$ is also not renormalizable for $a\in[0,1]$. 
Now, assume for the moment that $T_a$ is renormalizable for each $a\in[0,1]$ and consider the one-parameter 
family defined by $\tilde{T}_a=T_{\be(a)^2,\al(a)\be(a)}$. Note that if we show typicality of the turning point for the 
family $\tilde{T}_a$, for a.e. $a\in[0,1]$, this implies a.s. typicality of the turning point 
for the original family $T_a$. Furthermore, if condition~(I) is satisfied for the family $T_a$ together with a map $Y$ as in 
Corollary~\ref{c.tent1}, 
then it is easy to check that condition~(I) is also satisfied for the family $\tilde{T}_a$ together 
with the map $\tilde{Y}$ which is the map for the conjugated system corresponding to $Y$.
Since the $a$-derivative of $\al(a)\be(a)$ is positive and 
the $a$-derivative of $\be(a)^2$ is non-negative, the new one-parameter family $\tilde{T}_a$ 
fits into the family of skew tent maps described in the beginning of this section. 
Furthermore, it is known that for each $a\in[0,1]$, $T_a$ is 
at most a finite number of times renormalizable where this number is bounded above by a constant only 
dependent on $(\al_0,\be_0)$ and not on the parameter $a$ (this can easily be derived by looking, e.g., at the 
topological entropy of $T_a$, see \cite{mv} page 137). 
Altogether, we derive that in order to prove Theorem~\ref{t.tent} (and therewith also Corollary~\ref{c.tent1}) we can 
without loss of generality restrict ourself to 
the case when $T_a$, $a\in[0,1]$, is not renormalizable, i.e., we assume that 
\begin{equation}
\label{eq.nonrenorm}
\al_0>\be_0/(\be_0^2-1).
\end{equation}

Observe that it is only possible for the parameter 
$a=1$ to satisfy the equality $\al(a)^{-1}+\be(a)^{-1}=1$. Thus, since we are only interested 
in Lebesgue almost every parameter we can neglect skew tent maps whose slopes satisfy $\al^{-1}+\be^{-1}=1$, i.e., 
we assume that
\begin{equation}
\label{eq.notlog2}
\al_1^{-1}+\be_1^{-1}>1.
\end{equation}

\subsection{Proof of Theorem~\ref{t.tent}}
\label{ss.eins}
Since we can restrict ourself to skew tent maps which are non-renormalizable, this immediately implies that these maps are 
mixing. Hence, we can apply Theorem~\ref{t.main1} and in order to prove Theorem~\ref{t.tent} 
it is sufficient to show that there exists an iteration $j_0\ge3$ such that 
\begin{equation}
\label{eq.tentto}
|D_aT_a^{j_0}(0)|>\frac{\sup_{a\in[0,1]}\sup_{x\in[T_a(1),1]}\left|\partial_aT_a(x)\right|}{\la-1},
\end{equation}
for almost all $a\in[0,1]$. The main computation needed for the verification of \eref{eq.tentto} is already done in a paper 
by Misiurewicz and Visinescu \cite{mv} (see Lemma~3.3 and~3.4 therein; and Proposition~\ref{p.mv} below), 
where they show monotonicity of the 
kneading sequence for skew tent maps. (A reader not familiar with the basic notions and facts of kneading theory 
can find them in \cite{ce}.) The main result in \cite{mv} (see Theorem~A therein) is the following.

\begin{Thm}
\label{t.monotonicity}
Let $(\al,\be)$ and $(\al',\be')$ be in the set $\{(\al,\be)\ ;\ \al,\be>1\ \text{and}\ \al^{-1}+\be^{-1}\ge1\}$. 
If $(\al',\be')>(\al,\be)$ then the kneading sequence of $T_{\al',\be'}$ is strictly greater than the 
kneading sequence of $T_{\al,\be}$.
\end{Thm}

Since the derivatives of $\al(a)$ and $\be(a)$ are non-negative and at least one of them is positive, we obtain 
strict monotonicity of the kneading sequence for our family $T_a$, $a\in[0,1]$.
Let $j\ge1$. Note that if $T_a^i(0)=0$ for some $1\le i<j$, then the kneading sequence of $T_a$ ends with $C$ and has length 
smaller than $j$. By the strict monotonicity of the kneading sequence, $T_a$ can have 
such a kneading sequence only for finitely many parameter values $a\in[0,1]$. Hence, for each $j\ge1$, 
there are only a finite number of $a$ values such that $T_a^j(0)$ is not differentiable in $a$.
To establish \eref{eq.tentto} we will use some derivative estimates 
given in \cite{mv}. To this end we will look at the kneading sequences of the $T_a$'s. 
Every kneading sequence of a map $T_a$ in our family starts with $RL$ and is smaller or equal than the 
sequence $RL^{\infty}$. In fact, by \eref{eq.notlog2} and by the monotonicity of the kneading 
sequence, the kneading sequence of $T_1(=T_{\al_1,\be_1})$ is strictly smaller than $RL^{\infty}$. 
Let $1\le m_1<\infty$ be the integer such that the kneading sequence of $T_1$ starts with 
$RL^{m_1}R$ or is equal to $RL^{m_1}C$. From \cite{mv} we derive the following result.

\begin{Prop}
\label{p.mv}
There exists a constant $\ka>0$ such that for all $a\in[0,1]$ for which $T_a^j(0)$, $j\ge3$, is differentiable in $a$ we have
\begin{equation}
\label{eq.mv1}
\left|\partial_{\al}T_{\al(a),\be(a)}^{j}(0)\right|,\ 
\left|\partial_{\be}T_{\al(a),\be(a)}^{j}(0)\right|\ge\ka\be_0^{\left[\frac{j-3}{m_1}\right]},
\end{equation}
and, furthermore,
\begin{equation}
\label{eq.mv2}
\sign(\partial_{\al}T_{\al(a),\be(a)}^{j}(0))=\sign(\partial_{\be}T_{\al(a),\be(a)}^{j}(0))
=\sign(T_a^{j-1}\,'(1)).
\end{equation}
\end{Prop}

\begin{proof}
The proof of Proposition~\ref{p.mv} follows from Lemma~3.3 and~3.4 in \cite{mv}. For this note that for 
each $a\in[0,1]$, the integer $m\ge1$ such that the kneading sequence of $T_a$ starts with 
$RL^mR$ or is equal to $RL^mC$ is smaller or equal than $m_1$. 
Observe also that $x_j$ in \cite{mv} corresponds to $T_a^{j+1}(0)$ in our setting. 
Actually, Lemma~3.4 in \cite{mv} is only formulated for the case when $j-3\ge m$. But considering Lemma~3.4 i) 
in \cite{mv} it is easy to deduce that Proposition~\ref{p.mv} also holds when $3\le j<m+3$.
\end{proof}

For $j\ge3$, we have 
\begin{equation}
\label{eq.mv3}
D_aT_a^j(0)=\al'(a)\partial_{\al}T_{\al(a),\be(a)}^j(0)+\be'(a)\partial_{\be}T_{\al(a),\be(a)}^j(0),
\end{equation}
for all $a$ for which $a\mapsto T_a^j(0)$ is differentiable.
Since $\al'(a),\be'(a)\ge0$ and at least one of the derivatives $\al'(a)$ and $\be'(a)$ is uniformly bounded away from $0$, 
by ~\eref{eq.mv1} and ~\eref{eq.mv2}, $\inf_{a\in[0,1]}|D_aT_a^j(0)|$ is growing exponentially in $j$. 
Thus, we can fix an integer $j_0\ge3$ such that \eref{eq.tentto} is satisfied which concludes the proof of Theorem~\ref{t.tent}.


\section{Markov partition preserving one-parameter families}
\label{s.markov}

Assume that we have a one-parameter family $T_a:[0,1]\to[0,1]$, $a\in\I$, as described 
in Section~\ref{ss.prel} with a partition $0\equiv b_0(a)<b_1(a)<...<b_{p_0}(a)\equiv1$ and 
satisfying properties (i)-(iii). We require additionally that the family 
$T_a$ fulfills the following Markov property. Set $B_k(a)=(b_{k-1}(a),b_k(a))$, $1\le k\le p_0$.

\begin{itemize}
\item[(M)]
For each $1\le k\le p_0$ the image $T_a(B_k(a))$, $a\in\I$, is a union of monotonicity intervals $B_l(a)$, $1\le l\le p_0$ 
(modulo a finite number of points).
\end{itemize}

\begin{Thm}
\label{t.markov}
If the one-parameter family $T_a$, $a\in\I$, satisfies the Markov property (M) and if for a $C^1$ map 
$X:\I\to[0,1]$ condition~(I) is fulfilled, then $X(a)$ is typical for $\mu_a$, for a.e. $a\in\I$.
\end{Thm}

\begin{Exa}
\label{ex.markov}
Let 
$$
\T_a(x)=\left\{\begin{array}{ll}
              \frac{x}{a}\quad&\text{if}\ x<a,\\
              \frac{x-a}{1-a}\quad&\text{otherwise},
              \end{array}\right.
$$
and $g:[0,1]\to[0,1]$ a $C^{1,1}(L)$ homeomorphism such that $\inf_xg'(x)>0$ and such that the set
$$
\I=\{a\in(0,1)\ ;\ \inf_x \T_a'(g(x))g'(x)>1\}
$$
is non-empty. Clearly, $\I$ is an (open) interval. We define the one-parameter family $T_a:[0,1]\to[0,1]$ as
$$
T_a(x)=\T_a(g(x)),\qquad a\in\I.
$$
By \cite{wong}, since $T_a$ has only one point of discontinuity, there exists a unique a.c.i.p. $\mu_a$. 
From the verification of condition~(II) in the proof of Theorem~\ref{t.markov}, it will follow that $\supp(\mu_a)=[0,1]$. 

\begin{Prop}
\label{p.markov}
If $Y:\I\to(0,1)$ is a $C^1$ map such that $Y'(a)\le0$, then
$Y(a)$ is typical for $\mu_a$, for a.e. parameter $a\in\I$.
\end{Prop}

\begin{proof}
Le $\tilde{\I}\subset\I$ be a closed interval. 
It follows that there exist constants $1<\la\le\La<\infty$ which are uniform lower and upper 
bounds for the expansion in the family $T_a$, $a\in\tilde{\I}$ (cf. \eref{eq.la}).
Since the map $g$ is $C^{1,1}(L)$ it follows that there exists $\tilde{L}\ge L$ such that 
$T_a$ is piecewise $C^{1,1}(\tilde{L})$ for all $a\in\tilde{\I}$. 
One easily checks that the one-parameter family $T_a$, $a\in\tilde{\I}$, fits into the model described in 
Section~\ref{ss.prel} satisfying properties (i)-(iii). Hence, we can apply 
Theorem~\ref{t.markov} to this family. Clearly, $T_a$ satisfies the Markov property (M). 
In order to show a.s. typicality, it is only left to verify condition (I). 
We will use the criteria in Lemma~\ref{l.startcalc}. We have
$$
\partial_aT_a(x)=\left\{\begin{array}{ll}
              -\frac{g(x)}{a^2}\quad&\text{if}\ g(x)<a,\\
              -\frac{1+g(x)-2a}{(1-a)^2}\quad&\text{otherwise},
              \end{array}\right.
$$
which is non-positive for all $a\in\tilde{\I}$. Since $\tilde{\I}$ is a closed interval, the image of $\tilde{\I}$ by $Y$ has 
positive distance to $0$ and $1$ and, hence, there exists $\kappa>0$ such that $\sup_{a\in\tilde{\I}}\partial_aT_a(Y(a))\le-\kappa$. 
Recall the formula 
\eref{eq.startcalc1} for the derivative of $y_j$ (set $k=0$). All the terms in the right-hand side of \eref{eq.startcalc} are non-positive 
and the term $T_a^{j-1}\,'(y_1(a))\partial_aT_a(Y(a))$ is decreasing faster than $-\kappa\la^{j-1}$. Thus, we find $j_0\ge0$ such that 
\eref{eq.startcalc1} is satisfied for all parameter values $a\in\tilde{\I}$ in which $y_{j_0}$ is differentiable. 
In order to apply Lemma~\ref{l.startcalc},
it is only left to show that the number of 
$a\in\tilde{\I}$ in which $y_{j_0}$ is not differentiable is finite.
Since $Y'(a)\le0$ and since the point of discontinuity $g^{-1}(a)$ of $T_a$ is strictly increasing in $a$, there can only 
be one point in the interval $\tilde{\I}$ in which $y_1$ is not differentiable. 
Since $y_j'(a)\le0$, $j\ge1$, we can repeat this reasoning and
it follows that there 
are only finitely many points in which $y_{j_0}$ is not differentiable. Since we can cover $\I$ by a countable 
number of such intervals $\tilde{\I}$, this concludes the proof of the Proposition~\ref{p.markov}. 
\end{proof}
\end{Exa}

We turn to the proof of Theorem~\ref{t.markov}.
\begin{proof}
In order to proof Theorem~\ref{t.markov}, it is sufficient to verify conditions~(II) and (III). 
We consider first condition~(II). As in the first paragraph in Section~\ref{ss.zweib}, by Lemma~\ref{l.app1}, we can without loss 
of generality assume that there is a constant $C=C(\I)\ge1$ such that
for each $a\in\I$ the density $\varphi_a$ is bounded from above by $C$
and, further, there exists an interval $J(a)$ of length $C^{-1}$ such that $\varphi_a$ restricted 
to $J(a)$ is bounded from below by $C^{-1}$. 
Since for each $a\in\I$ the expansion of $T_a$ is at least 
$\la$, we derive that there is an integer $i\ge1$ independent on $a$ such that the number of elements in 
$\P_i|J(a)$ is greater or equal than $3$. By (M), we derive that the image by $T_a^{i-1}$ of an element $\om\in\P_i|J(a)$, 
which is not adjacent to a boundary point of 
$J(a)$, is a monotonicity interval $B_l(a)$, $1\le l\le p_0$. By our assumption on the 
one-parameter family $T_a$, the measure $\mu_a$ is ergodic. It follows that there is an integer $i'\ge1$ such that 
$|K(a)\setminus T_a^{i'}(B_l(a))|=0$. (Obviously $i'$ can be chosen independently on $a\in\I$ and the monotonicity interval 
$B_l(a)\subset K(a)$.)
Thus, setting $j=i+i'$ for almost every $y\in K(a)$, there exists a point $x\in J(a)$ such 
that $x$ is mapped to $y$ after $j$ iterations, i.e., $T_a^j(x)=y$. Now, inequality~\eref{eq.perron22} provides 
us with a lower bound for the density. Note that from this argument follows 
that $\supp(\mu_a)=[0,1]$ in Example~\ref{ex.markov}.

To verify (III) we observe that, since in this Markov setting 
$K(a)$, $a\in\I$, is the union of monotonicity intervals $B_l(a)$, $1\le l\le p_0$, there exists even a bijection
$$
\S_{a_1,a_2,j}:\P_j(a_1)\to\P_j(a_2),
$$
for all $a_1,a_2\in\I$ and $j\ge1$, satisfying \eref{eq.drei1}. Since for each element $\om\in\P_j(a)$ the image $T_a^j(\om)$ 
is a union of monotonicity intervals $B_l(a)$, $1\le l\le p_0$, and since 
by property~(i) in Section~\ref{ss.prel} the boundary points of $B_l(a)$ are Lipschitz continuous in $a$ and $|B_l(a)|\ge\delta_0$,
we get that also \eref{eq.dreidist} and \eref{eq.drei2} are satisfied where we can take $C_2=\max\{L,\delta_0^{-1}\}$.
\end{proof}

\appendix
\section{}
\label{A:appendix}

\begin{Lem}
\label{l.app1}
Let $T_a:[0,1]\to[0,1]$, $a\in\I$, be a one-parameter family as described in Section~\ref{ss.prel}, satisfying 
properties (i)-(iii) and condition~(III). Disregarding a finite number of parameters in $\I$, we can cover $\I$ by a countable 
number of intervals $\tilde{\I}\subset\I$ such that on each interval $\tilde{\I}$ the following holds. 
There exists a constant $C=C(\tilde{\I})\ge1$ such that for each $a\in\tilde{\I}$ the density $\varphi_a$ of $\mu_a$ 
is bounded above by $C$ and, further, there exists an interval $J(a)\subset[0,1]$ of size $C^{-1}$ 
such that $\varphi_a$ restricted to $J(a)$ is bounded from below by $C^{-1}$.
\end{Lem}

\begin{proof}
For each $a\in\I$ it follows from \cite{wong} p.496 line 5 and \cite{lasota} p.484 line 6, that the variation over
the unit interval of the density $\varphi_a$ is bounded above by a constant 
$$
C_v(a)=\frac{3}{\delta(a)(\la^\tau-3)},
$$ 
where the integer $\tau\ge1$ is chosen so large that
$3/\la^\tau<1$ and the number $\delta(a)$ is given by 
$$
\delta(a)=\min\{|\om|\ ;\ \om\in\P_\tau(a)\}>0.
$$
(In \cite{wong} and \cite{lasota} $\delta(a)$ is the minimal size of the monotonicity intervals for the map $T_a^\tau:[0,1]\to[0,1]$. 
But since the elements of $\P_\tau(a)$ are monotonicity intervals for the map $T_a^\tau:K(a)\to K(a)$, the constant $C_v(a)$ 
is greater or equal than the corresponding constant in \cite{wong} and \cite{lasota}.)

\begin{Claim}
For $j\ge1$, let $(s_0,...,s_{j-1})$ be a sequence of symbols $s_i\in\{1,...,p_1\}$, 
$0\le i<j$. If $a_0\in\I$ is a parameter value such that there exists an element $\om(a_0)\in\P_j(a_0)$ 
satisfying
$$
\ind_{a_0}(T_{a_0}^i(\om(a_0)))=s_i,\qquad0\le i<j,
$$
then there is a neighborhood $U$ of $a_0$ in $\I$ such that for all $a\in U$ there is an element 
$\om(a)\in\P_j(a)$ having the same combinatorics as $\om(a_0)$, i.e., $\ind_a(T_a^i(\om(a)))=s_i$, 
$0\le i<j$. Furthermore, the boundary points of  $T_a^j(\om(a))$ depend continuously on $a\in U$.
\end{Claim}

\begin{proof}
We prove the claim by induction over $j\ge1$. We do not make use of condition~(III). 
For $j=1$ the elements in $\P_1(a)$ corresponding to the symbols $s_0\in\{1,...,p_1\}$ 
are the intervals $D_k(a)$, $1\le k\le p_1$. The boundary points of these intervals are, by property~(iii), 
continuous functions on $\I$. 
Using properties~(i) and (ii), one can easily show that the boundary points of $T_a(D_k(a))$ are continuous 
on $\I$. Now, assume that the statement holds for some $j\ge1$. Fix a sequence $(s_0,...,s_j)$ of symbols in 
$\{1,...,p_1\}$. Let $a_0\in\I$ be a parameter such that there exists an element $\om(a_0)\in\P_{j+1}(a_0)$ 
satisfying $\ind_{a_0}(T_{a_0}^i(\om(a_0)))=s_i$, for all $0\le i<j+1$ (if there is no such a parameter $a_0$ 
for which the element $\om(a_0)$ exists then there is nothing to show). 
Let $\ome(a_0)\in\P_j(a_0)$ be the element containing $\om(a_0)$. By the induction 
assumption there exists a neighborhood $V$ of $a_0$ in $\I$ such that for all $a\in V$ there is an element 
$\ome(a)\in\P_j(a)$ having the same combinatorics as $\ome(a_0)$ and the boundary points of $\ome(a)$ 
and $T_a^j(\ome(a))$ depend continuously on $a\in V$. Note that if $y(a_0)$ is a boundary point of 
$T_{a_0}^j(\om(a_0))$ then it is either equal to a partition point $b_k(a_0)$, $0\le k\le p_0$, or it is a boundary point of 
$T_{a_0}^j(\ome(a_0))$. By the continuity of the boundary points of $T_a^j(\ome(a))$ on $V$ and the continuity of 
$a\mapsto b_k(a)$, we deduce that there exists a neighborhood $U\subset V$ of $a_0$ in $\I$ such 
that for each $a\in U$ there exists an element $\om(a)\in\P_{j+1}(a)$ having the same combinatorics as $\om(a_0)$. 
Since the boundary points of $T_a^j(\om(a))$ are continuous on $U$, we can once more apply properties~(i) and (ii) to 
deduce that also the boundary points of $T_a^{j+1}(\om(a))$ are continuous on $U$. 
\end{proof}

Let $(s_0,...,s_{j-1})$ be a sequence of symbols $s_i\in\{1,...,p_1\}$. If for a parameter $a_0\in\I$ there exists 
an element $\om(a_0)\in\P_j(a_0)$ which corresponds to this sequence of symbols then, by condition~(III), 
for each $a\ge a_0$ there exists an element $\om(a)\in\P_j(a)$ corresponding to this sequence of symbols. 
Furthermore, $|T_a^j(\om(a))|/\Lambda^j$ is a lower bound for the size of this element which is by the claim above 
continuous in $a$. This implies that there is a map $\tilde{\delta}:\I\to(0,1]$ which is piecewise continuous with only 
a finite number of discontinuities and $\delta(a)\ge\tilde{\delta}(a)$.
Hence, disregarding a finite number of parameter values in $\I$, we can cover $\I$ by a countable number of intervals $\tilde{\I}\subset\I$ 
such that for each such interval $\tilde{\I}$ there is a constant $\delta_0=\delta_0(\tilde{\I})>0$ such that 
\begin{equation}
\label{eq.kaap}
\delta(a)\ge\delta_0,
\end{equation} 
for all $a\in\tilde{\I}$. It follows that that there is a constant $C_v=C_v(\tilde{\I})\ge1$ 
such that the variation of $\varphi_a$ is bounded from above by $C_v$ for all $a\in\tilde{\I}$. 
Since $\int_0^1\varphi_a(x)dx=1$, this immediately implies that $\varphi_a$ is bounded from above by $C_v+1$. 
To establish a lower bound on a subinterval of $K(a)$, we observe the following.

\begin{Claim}
If the variation over $[0,1]$ of a function $\varphi:[0,1]\to\real_+$
is bounded from above by a constant $C_v\ge1$, and if
$\int_0^1\varphi(x)dx=1$, then there exists an interval $J$ of length
$1/2C_v$ such that $\varphi(x)\ge1/3C_v$ for all $x\in J$.
\end{Claim}

\begin{proof}
Let $N=[2C_v]$, divide the unit interval into $N$ disjoint intervals
$J_1,...,J_N$ of length $1/N$, and, for $1\le l\le N$, set 
$m_l=\inf\{\varphi(x)\ ;\ x\in J_l\}$ and 
$M_l=\sup\{\varphi(x)\ ;\ x\in J_l\}$. Since $1=\int_0^1\varphi(x)dx\le\sum_{l=1}^NM_l/N$, it
follows that $N\le\sum_{l=1}^NM_l$. If $m_l<1/3C_v$, for all
$1\le l\le N$, it would follow that the variation of
$\varphi$ is strictly greater than $\sum_{l=1}^N(M_l-1/3C_v)\ge
N(1-1/3C_v)\ge C_v$, where the last inequality follows since
$C_v\ge1$. Hence, at least for one $1\le l\le N$, $m_l\ge1/3C_v$.
\end{proof}

Setting $C=3C_v$ this concludes the proof of Lemma~\ref{l.app1}.
\end{proof}

\section*{Acknowledgement}
I am grateful to M. Benedicks and K. Bjerkl\"ov for many fruitful discussions and for encouraging and supporting 
me during the writing of this paper. A part of this paper was written at the Institut Mittag-Leffler (Djursholm, Sweden).
I thank Tomas Persson for pointing out that 
in condition~(II) it is sufficient to require that the constant $C_2$ depends measurably on the parameter $a$ (cf. Remark~\ref{r.tomas}).

\end{document}